\documentclass{amsart}

\usepackage[usenames,dvipsnames]{xcolor}
\usepackage{amsthm}
\usepackage{amssymb}
\usepackage{amsfonts}
\usepackage{fancyhdr}
\usepackage{tikz}
\usepackage{hyperref}
\usepackage{graphicx}
\usepackage{caption}
\usepackage[noadjust]{cite}
\usepackage{subcaption}
\usepackage{amsmath,amscd}
\usepackage{float}
\usepackage{mathrsfs}
\usetikzlibrary{patterns}
\usepackage[mathscr]{euscript}
\usepackage{adjustbox}

\newtheorem{theorem}{Theorem}[section]

\newtheorem{lemma}[theorem]{Lemma}
\newtheorem{proposition}[theorem]{Proposition}

\theoremstyle{definition}

\newtheorem{definition}[theorem]{Definition}
\newtheorem{remark}[theorem]{Remark}

\definecolor{light-gray}{gray}{0.8}
\definecolor{mgray}{gray}{.6}

\begin{document}
\date{}
\title{Baxter Posets}
\thanks{The author was partially supported by NSF grant DMS-1500949.} 
\author{Emily Meehan}

\maketitle

\begin{abstract}
We define a family of combinatorial objects, which we call Baxter posets.
We prove that Baxter posets are counted by the Baxter numbers by showing that they are the adjacency posets of diagonal rectangulations.
Given a diagonal rectangulation, we describe the cover relations in the associated Baxter poset.
Given a Baxter poset, we describe a method for obtaining the associated Baxter permutation and the associated twisted Baxter permutation.
\end{abstract}

\tableofcontents

\section{Introduction}\label{intro}
The Baxter numbers
$$B(n)={{n+1}\choose 1}^{-1} {{n+1} \choose 2}^{-1} \sum_{k=1}^n {n+1 \choose k-1} {n+1 \choose k} {n+1 \choose k+1}$$
count Baxter permutations \cite{chung}, twisted Baxter permutations \cite{drec}, certain triples of non-intersecting lattice paths \cite{dulucq}, noncrossing arc diagrams consisting of only left and right arcs  \cite{Noncrossing}, certain Young tableaux \cite{dulucq2}, twin binary trees \cite{dulucq2}, diagonal rectangulations~\cite{Ackerman2, felsner, drec}, and other families of combinatorial objects.

In this paper, we define Baxter posets and prove that they are also counted by the Baxter numbers.
Baxter posets are closely related to Catalan combinatorics.
Specifically, Baxter posets (and the closely related diagonal rectangulations) can be realized through ``twin" Catalan objects.
Additionally, the relationship between Baxter posets and diagonal rectangulations is analogous to the relationship between two Catalan objects, specifically sub-binary trees and triangulations of convex polygons.
As a prelude to our discussion of Baxter posets, we describe a few Catalan objects and bijections between them.

Let $S_n$ denote the set of permutations of $[n]=\{1,\ldots,n\}$.
We say that $\sigma=\sigma_1\cdots \sigma_n \in S_n$ \emph{avoids the pattern 2-31} if there does not exist $i<j$ such that $\sigma_{j+1}<\sigma_i<\sigma_j$.
The Catalan number $C(n)=\frac{1}{n}{{2n}\choose{n}}$ counts the elements of~$S_n$ that avoid the pattern 2-31.
The map~$\tau_b$, described below and illustrated in Figure~\ref{tau b}, assigns a triangulation of a convex $(n+2)$-gon to each element of $S_n$, and restricts to a bijection between permutations that avoid 2-31 and triangulations of polygons.
Let $\sigma=\sigma_1\cdots \sigma_n \in S_n$ and let $P$ be a convex $(n+2)$-gon.
For convenience, deform~$P$ so that $P$ is inscribed in the upper half of a circle, and label each vertex of $P$, in numerical order from left to right, with an element of the sequence $0,1,\ldots,n+1$.
For each $i\in \{0,\ldots,n\}$, construct a path $P_i$ from the vertex labeled $0$ to the vertex labeled $n+1$ that visits the vertices labeled by elements of $\{\sigma_1,\ldots,\sigma_i\}$ in numerical order.
The union of these paths defines $\tau_b(\sigma)$, a triangulation of $P$.

Given a triangulation $\Delta$ of a convex $(n+2)$-gon $P$, deform $P$ (and $\Delta$) as above.
Construct a graph with an edge crossing each edge of $\Delta$ except the horizontal diameter, as shown in red in the left diagram of Figure \ref{tau b}.  
(This is essentially the dual graph of $\Delta$.)
In what follows, we will call this the \emph{dual graph construction}.
Terminology for the resulting family of trees is mixed in the literature, with adjectives such as complete, planar, rooted, and binary appearing inconsistently. 
We will call the resulting tree a binary tree and provide a careful definition.
For us, a \emph{binary tree} is a rooted tree such that every non-leaf has exactly two children, with one child identified as the left child and the other as the right child.
The dual graph construction gives a bijection between triangulations of a convex $(n+2)$-gon and binary trees with~$2n+1$ vertices.
The root of the binary tree corresponds to the bottom triangle of $\Delta$ and children are identified as left or right according to the embedding of $\Delta$ in the plane.
For a reason that will become apparent later, we deform each binary tree resulting from this bijection as shown in the right diagram of Figure \ref{tau b} so that the root is the lowermost vertex.
Removing the leaves of a binary tree and retaining the left-right labeling of each child, we obtain a \emph{sub-binary tree}, a rooted tree in which every vertex has 0, 1, or~2 children, and each child is labeled left or right, with at most one child of each vertex receiving each label.
The leaf-removal map is a bijection between binary trees with $2n+1$ vertices and sub-binary trees with $n$ vertices.
In the example shown in Figure \ref{tau b}, the edges removed by this map are shown as dashed segments.

\begin{figure}
\begin{tabular}{c  c  c}
\adjustbox{valign=c}{
\begin{tikzpicture}[scale=.65]
\foreach \a in {0,1,...,9}{
\draw (-\a*180/9+180: 3.8cm) node{ \a};
}
\foreach \b in {0,1,...,9}{
\draw[fill] (-\b*180/9+180: 3.5cm) circle (.2ex);
}
\draw(0:3.5)--(180/9: 3.5)--(180*2/9: 3.5)--(180*3/9: 3.5)--(180*4/9: 3.5)--(180*5/9: 3.5)--(180*6/9: 3.5)--(180*7/9: 3.5)--(180*8/9: 3.5)--(180*9/9: 3.5);
\draw(180:3.5)--(180*7/9:3.5)--(180*4/9:3.5)--(180*2/9:3.5)--(0:3.5);
\draw(180*7/9:3.5)--(180*5/9:3.5);
\draw (180: 3.5)--(180*4/9:3.5)--(0:3.5)--cycle;

\foreach \b in {0,1,...,8}{
\draw[fill, red] (\b*180/9+180/18: 3.5cm) circle (.2ex);
}
\draw[red](80:.5)--(40:3)--(20:3.35)--(10:3.5);
\draw[red](20:3.35)--(30:3.5);
\draw[red](40:3)--(60:3.35)--(50: 3.5);
\draw[red](60:3.35)--(70: 3.5);
\draw[red](80: .5)--(140:2.8)--(160: 3.35)--(170: 3.5);
\draw[red](140:2.8)--(100: 3.3)--(120:3.37)--(130:3.5);
\draw[red](100: 3.3)--(90: 3.5);
\draw[red](120:3.37)--(110:3.5);
\draw[red](160:3.35)--(150: 3.5);

\draw[fill, red] (80:.5) circle (.2ex);
\draw[fill, red] (40:3) circle (.2ex);
\draw[fill, red] (20:3.35) circle (.2ex);
\draw[fill, red] (60:3.35) circle (.2ex);
\draw[fill, red] (140:2.8) circle (.2ex);
\draw[fill, red] (160:3.35) circle (.2ex);
\draw[fill, red] (100:3.3) circle (.2ex);
\draw[fill, red] (120:3.37) circle (.2ex);
\end{tikzpicture}}

&  &

\adjustbox{valign=c}{\begin{tikzpicture}[scale=.5]
\draw[red,dashed](-4,4)--(-3.5, 3.5);
\draw[red, dashed](3.5, 3.5)--(4,4);
\draw[red](-3.5, 3.5)--(0,0)--(3.5,3.5);
\draw[red,dashed](-3,4)--(-3.5, 3.5);
\draw[red](-1,3)--(-2,2);
\draw[red,dashed](0,4)--(-1,3);
\draw[red,dashed](-2,4)--(-1.5, 3.5);
\draw[red](-1.5,3.5)--(-1,3);
\draw[red,dashed](-1,4)--(-1.5, 3.5);
\draw[red](1.5, 3.5)--(2.5,2.5);
\draw[red,dashed](1,4)--(1.5,3.5);
\draw[red,dashed](2,4)--(1.5, 3.5);
\draw[red,dashed](3,4)--(3.5, 3.5);

\draw[fill, red] (0,0) circle (.2ex); 
\draw[fill, red] (-3.5, 3.5) circle (.2ex); 
\draw[fill, red] (-2,2) circle (.2ex);
\draw[fill, red] (-1,3) circle (.2ex);
\draw[fill, red] (-1.5,3.5) circle (.2ex);
\draw[fill, red] (2.5, 2.5) circle (.2ex);
\draw[fill, red] (1.5, 3.5) circle (.2ex);
\draw[fill, red] (3.5,3.5) circle (.2ex);
\draw[fill, red] (-4,4) circle (.2ex);
\draw[fill, red] (-3,4) circle (.2ex);
\draw[fill, red] (-2,4) circle (.2ex);
\draw[fill, red] (-1,4) circle (.2ex);
\draw[fill, red] (0,4) circle (.2ex);
\draw[fill, red] (1,4) circle (.2ex);
\draw[fill, red] (2,4) circle (.2ex);
\draw[fill, red] (3,4) circle (.2ex);
\draw[fill, red] (4,4) circle (.2ex);
\end{tikzpicture}}
\end{tabular}\caption{The Catalan objects obtained by applying the described bijections to the 2-31 avoiding permutation~$52143768$. 
}\label{tau b}
\end{figure}
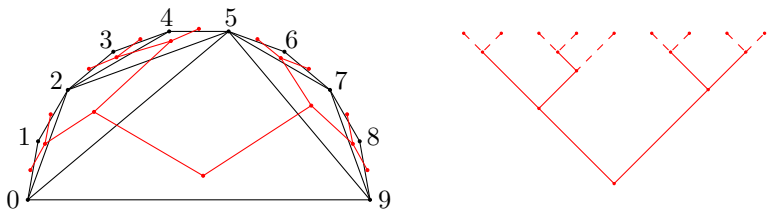

We will make use of a second similar map from permutations to triangulations.
We say that a permutation \emph{avoids the pattern 31-2} if there does not exist ${i<i+1<j}$ such that $\sigma_{i+1}<\sigma_j<\sigma_i$.
The map $\tau_t$ described below restricts to a bijection between elements of $S_n$ that avoid 31-2 and triangulations of a convex $(n+2)$-gon.
Let $\sigma \in S_n$ and $P$ a convex $(n+2)$-gon.
Deform $P$ and label its vertices as shown in the example in Figure \ref{tau t}.
For each $i\in \{0,1,\ldots,n\}$, construct the path~$P_i$ that begins at the vertex labeled~0, visits in numerical order each vertex labeled by an element of $[n]-\{\sigma_1,\ldots,\sigma_i\}$, and ends at the vertex labeled $n+1$.
The union of these paths is $\tau_t(\sigma)$.
Performing the dual graph construction and then the leaf-removal map, we obtain corresponding binary and sub-binary trees.  
This time, we choose to deform the binary and sub-binary trees so that the root is the uppermost vertex, as illustrated in the right diagram of Figure \ref{tau t}.

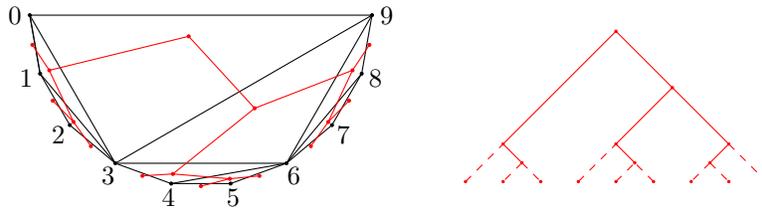
\begin{figure}
\begin{tabular}{c c c}
\adjustbox{valign=c}{
\begin{tikzpicture}[scale=.65]
\foreach \a in {0,1,...,9}{
\draw (\a*180/9+180: 3.8cm) node{ \a};
}
\foreach \b in {0,1,...,9}{
\draw[fill] (\b*180/9+180: 3.5cm) circle (.2ex);
}
\draw(0:3.5)--(-180/9: 3.5)--(-180*2/9: 3.5)--(-180*3/9: 3.5)--(-180*4/9: 3.5)--(-180*5/9: 3.5)--(-180*6/9: 3.5)--(-180*7/9: 3.5)--(-180*8/9: 3.5)--(-180*9/9: 3.5);
\draw(180:3.5)--(-160:3.5)--(-120:3.5)--(-60:3.5)--(0:3.5);
\draw(-60:3.5)--(-100:3.5);
\draw(-60: 3.5)--(-20:3.5);
\draw (180: 3.5)--(-120:3.5)--(0:3.5)--cycle;

\foreach \b in {0,1,...,8}{
\draw[fill, red] (-\b*180/9-180/18: 3.5cm) circle (.2ex);
}
\draw[red](-120:.5)--(-60: 2.2)--(-100:3.3)--(-110:3.5);
\draw[red](-100:3.3)--(-80: 3.4)--(-90: 3.5);
\draw[red](-80: 3.4)--(-70:3.5);
\draw[red](-60:2.2)--(-20:3.3)--(-40: 3.4)--(-50: 3.5);
\draw[red](-40:3.4)--(-30: 3.5);
\draw[red](-20:3.3)--(-10:3.5);
\draw[red](-120:.5)--(-160:3.3)--(-170:3.5);
\draw[red](-160:3.3)--(-140:3.4)--(-150: 3.5);
\draw[red](-140: 3.4)--(-130:3.5);

\draw[fill, red] (-120:.5) circle (.2ex);
\draw[fill, red] (-60: 2.2) circle (.2ex);
\draw[fill, red] (-100: 3.3) circle (.2ex);
\draw[fill, red] (-80: 3.4) circle (.2ex);
\draw[fill, red] (-20: 3.3) circle (.2ex);
\draw[fill, red] (-40:3.4) circle (.2ex);
\draw[fill, red] (-160: 3.3) circle (.2ex);
\draw[fill, red] (-140:3.4) circle (.2ex);
\end{tikzpicture}}

& & 

\adjustbox{valign=c}{\begin{tikzpicture}[scale=.5]
\draw[red,dashed](0,0)--(1,1);
\draw[red, dashed](7,1)--(8,0);
\draw[red](1,1)--(4,4)--(7,1);
\draw[red](1,1)--(1.5,.5);
\draw[red,dashed](1.5,.5)--(1,0);
\draw[red,dashed](1.5, .5)--(2,0);
\draw[red](7,1)--(6.5, .5);
\draw[red,dashed](6.5, .5)--(7,0);
\draw[red,dashed](6.5, .5)--(6,0);
\draw[red,dashed](3,0)--(4,1);
\draw[red](4,1)--(5.5, 2.5);
\draw[red,dashed](4,0)--(4.5, .5);
\draw[red,dashed](5,0)--(4.5, .5);
\draw[red](4.5, .5)--(4, 1);

\draw[fill, red] (1,1) circle (.2ex);
\draw[fill, red] (7,1) circle (.2ex);
\draw[fill, red] (4,4) circle (.2ex); 
\draw[fill, red] (1.5, .5) circle (.2ex); 
\draw[fill, red] (6.5, .5) circle (.2ex); 
\draw[fill, red] (4,1) circle (.2ex);
\draw[fill, red] (5.5, 2.5) circle (.2ex); 
\draw[fill, red] (4.5, .5) circle (.2ex); 
\draw[fill, red] (0,0) circle (.2ex);
\draw[fill, red] (1,0) circle (.2ex);
\draw[fill, red] (2,0) circle (.2ex);
\draw[fill, red] (3,0) circle (.2ex);
\draw[fill, red] (4,0) circle (.2ex);
\draw[fill, red] (5,0) circle (.2ex);
\draw[fill, red] (6,0) circle (.2ex);
\draw[fill, red] (7,0) circle (.2ex);
\draw[fill, red] (8,0) circle (.2ex);
\end{tikzpicture}}
\end{tabular}\caption{The Catalan objects obtained by applying the described bijections to the 31-2 avoiding permutation $21547863$.
}\label{tau t}
\end{figure}

Although a sub-binary tree is an unlabeled graph, for each sub-binary tree with~$n$ vertices, there exists a unique labeling of its vertices by the elements of $[n]$ such that every parent vertex has a label numerically larger than the labels of its left descendants and numerically smaller than its right descendants.
An example of a sub-binary tree with such a labeling is show in Figure \ref{treelabel}.
Let $T$ be a labeled sub-binary tree embedded in the plane as shown in Figure  \ref{treelabel} and $\Delta_T$ the associated triangulation.
View $T$ as the Hasse diagram of a poset.
We say that a total order~$L$ of the elements of $T$ is a \emph{linear extension} of $T$ if $x<_Ty$ implies that $x<_L y$.
The linear extensions of~$T$, viewed as permutations in one-line notation, are exactly the permutations that map to~$\Delta_T$ under $\tau_b$.
To see why, label each triangle of~$\Delta_T$ according to the label of its middle (from left to right) vertex, as illustrated in Figure \ref{treelabel}.
The linear extensions of $T$ are exactly the permutations that map to $\Delta_T$ because $x<_Ty$ if and only if the triangle labeled~$y$ is ``above" the triangle labeled~$x$. 
Similarly, given a sub-binary tree~$T'$, embedded in the plane as illustrated in Figure~\ref{tau t}, and associated triangulation $\Delta_{T'}$, we obtain a labeling of~$T'$ such that the linear extensions of~$T'$ are exactly the permutations that map to $\Delta_{T'}$ under~$\tau_t$. 

\begin{figure}
\begin{tabular}{c  c  c}
\adjustbox{valign=c}{
\begin{tikzpicture}[scale=.95]

\draw(0:3.5)--(180/9: 3.5)--(180*2/9: 3.5)--(180*3/9: 3.5)--(180*4/9: 3.5)--(180*5/9: 3.5)--(180*6/9: 3.5)--(180*7/9: 3.5)--(180*8/9: 3.5)--(180*9/9: 3.5);
\draw(180:3.5)--(180*7/9:3.5)--(180*4/9:3.5)--(180*2/9:3.5)--(0:3.5);
\draw(180*7/9:3.5)--(180*5/9:3.5);
\draw (180: 3.5)--(180*4/9:3.5)--(0:3.5)--cycle;
\node at (80: 1.5) {\tiny 5};
\node at (40: 3.1) {\tiny  7};
\node at (140: 2.8) {\tiny  2};
\node at (20: 3.4) {\tiny 8};
\node at (60: 3.4) {\tiny 6};
\node at (100: 3.3) {\tiny 4};
\node at (120: 3.4) {\tiny 3};
\node at (160: 3.38) {\tiny 1};
\end{tikzpicture}}

&  &

\adjustbox{valign=c}{
\begin{tikzpicture}[scale=.5]
\draw[red](-3.5, 3.5)--(0,0)--(3.5,3.5);
\draw[red](-1,3)--(-2,2);
\draw[red](-1.5,3.5)--(-1,3);
\draw[red](1.5, 3.5)--(2.5,2.5);

\draw[fill, red] (0,0) circle (.2ex); \node[red] at (0,-.4) {5};
\draw[fill, red] (-3.5, 3.5) circle (.2ex); \node[red] at (-3.6,3.9) {1};
\draw[fill, red] (-2,2) circle (.2ex); \node[red] at (-2.2,1.7) {2};
\draw[fill, red] (-1,3) circle (.2ex); \node[red] at (-.7,3) {4};
\draw[fill, red] (-1.5,3.5) circle (.2ex); \node[red] at (-1.6, 3.9) {3};
\draw[fill, red] (2.5, 2.5) circle (.2ex); \node[red] at (2.7,2.2) {7};
\draw[fill, red] (1.5, 3.5) circle (.2ex); \node[red] at (1.4,3.9) {6};
\draw[fill, red] (3.5,3.5) circle (.2ex); \node[red] at (3.6,3.9) {8};
\end{tikzpicture}}
\end{tabular}\caption{The labeling of a sub-binary tree obtained from the labeling of the triangles in the corresponding triangulation.}\label{treelabel}
\end{figure}
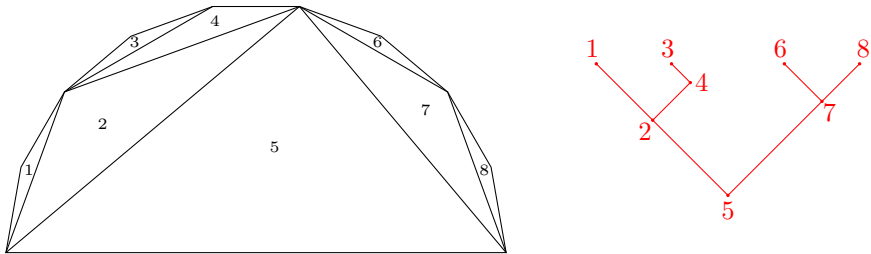

We now relate the Catalan objects described above to Baxter objects.
Specifically, we will see that diagonal rectangulations are made by gluing together binary trees, and we will construct Baxter posets so that they play the same role for diagonal rectangulations that sub-binary trees play for triangulations.

A permutation $\sigma=\sigma_1\cdots \sigma_n\in S_n$ \emph{avoids the patterns 2-41-3 and 3-41-2} if there does not exist $i<j<j+1<k$ such that $\sigma_{j+1}<\sigma_i<\sigma_k<\sigma_j$ or $\sigma_{j+1}<\sigma_k<\sigma_i<\sigma_j$. 
If $\sigma$ avoids the patterns 2-41-3 and 3-41-2, then we say that $\sigma$ is a \emph{twisted Baxter permutation}.
The twisted Baxter permutations in $S_n$ are counted by $B(n)$. 

Twisted Baxter permutations are related to certain decompositions of a square into rectangles.
Given $\sigma \in S_n$, glue the binary trees corresponding to~$\tau_b(\sigma)$ and~$\tau_t(\sigma)$, called \emph{twin binary trees}, along their leaves to obtain a decomposition of a square into $n$ rectangles, and then rotate the resulting figure $\pi/4$ radians clockwise.
The result of applying this binary tree gluing map to the permutation $52147862$ is shown in the left diagram of Figure \ref{rect}.
The binary trees which are glued together in this example are shown in Figures \ref{tau b} and \ref{tau t}.
We call each decomposition resulting from this binary tree gluing map a \emph{diagonal rectangulation} (defined precisely in Section \ref{definitions}) because the top-left to bottom-right diagonal of the square contains an interior point of each rectangle of the decomposition.
The map restricts to a bijection between twisted Baxter permutations and diagonal rectangulations. 

Given a diagonal rectangulation, label the rectangles of the decomposition according to the order in which they appear along the diagonal, labeling the upper-leftmost rectangle with 1 and the lower-rightmost rectangle with $n$.
We refer to the rectangle with label $i$ as ``rectangle~$i$."
Construct a poset $P$ on $[n]$ by declaring $x<_P y$ if the interior of the bottom or left side of rectangle~$y$ intersects the interior of the top or right side of rectangle~$x$, and then taking the reflexive and transitive closure of these relations.
Remark 6.7 in \cite{drec} explains that, before taking the reflexive and transitive closure, these relations are acyclic.
Thus $P$ is a partial order on~$[n]$.
This poset, which we call the \emph{adjacency poset} of the diagonal rectangulation, is defined in \cite{felsner, drec}.
(A more general set of posets, corresponding to elements of the Baxter monoid, are defined in ~\cite
{giraudo}.)
Each adjacency poset captures the ``right of" and ``above" relations of the diagonal rectangulation just as each sub-binary tree captures the ``above" relations of the corresponding triangulation.
Additionally, given an adjacency poset $P$ and the corresponding diagonal rectangulation $D$, the set of linear extensions of $P$ is the set of permutations that map to $D$ under the binary tree gluing map \cite[Remark~6.7]{drec}.
We note that two permutations $\sigma$ and $\psi$ map to the same diagonal rectangulation if and only if $\tau_b(\sigma)=\tau_b(\psi)$ and $\tau_t(\sigma)=\tau_t(\psi).$
Thus, the set of linear extensions of the adjacency poset of a diagonal rectangulation is the intersection of the sets of linear extensions of the labeled sub-binary trees obtained from $\tau_b$ and $\tau_t$.

As a diagonal rectangulation can be constructed from twin binary trees, the adjacency poset of a diagonal rectangulation can be constructed using the corresponding labeled sub-binary trees.
Let $D$ be a diagonal rectangulation, $P$ the associated adjacency poset, and~$T_b$ and $T_t$ respectively denote the corresponding labeled sub-binary trees obtained from $\tau_b$ and~$\tau_t$.
By declaring $x<_Py$ if $x<_{T_b} y$ or $x<_{T_t} y$ and then taking the transitive closure, we obtain all of the relations of~$P$.
Although it is simple to use the relations of $T_b$ and $T_t$ to list the relations of~$P$, it is not so straightforward to obtain a description of the Hasse diagram of~$P$ or to characterize the set of adjacency posets of diagonal rectangulations.
In any poset~$P$, we say that \emph{$x$ covers $y$}, denoted $x \lessdot_P y$, if $x<_P y$ and there exists no~$z$ such that $x<_P z <_P y$.
In Theorem~\ref{thm:adjcovers}, the first main result of this paper, we show that $x$ covers $y$ in the adjacency poset $P$ 
if and only if, in the associated diagonal rectangulation, rectangles $x$ and $y$ form one of the configurations shown in Figure \ref{fig:adjcovers}.
This theorem allows us to obtain a Hasse diagram for the adjacency poset from a diagonal rectangulation just as we easily obtain a sub-binary tree from a triangulation.

\begin{figure}
\begin{tabular}{c c c c c}
\adjustbox{valign=c}{
\begin{tikzpicture}[scale=.3]
\draw[fill] (0,8) circle (.3ex);
\draw[fill] (1,7) circle (.3ex);
\draw[fill] (2,6) circle (.3ex);
\draw[fill] (3,5) circle (.3ex);
\draw[fill] (4,4) circle (.3ex);
\draw[fill] (5,3) circle (.3ex);
\draw[fill] (6,2) circle (.3ex);
\draw[fill] (7,1) circle (.3ex);
\draw[fill] (8,0) circle (.3ex);
\draw[fill] (0,7) circle (.3ex);
\draw[fill] (2,8) circle (.3ex);
\draw[fill] (8,8) circle (.3ex);
\draw[fill] (8,2) circle (.3ex);
\draw[fill] (7,0) circle (.3ex);
\draw[fill] (0,0) circle (.3ex);
\draw[fill] (0,4) circle (.3ex);
\draw[fill] (2,4) circle (.3ex);
\draw[fill] (5,4) circle (.3ex);
\draw[fill] (5,2) circle (.3ex);
\draw[fill] (5,0) circle (.3ex);
\draw[fill] (2,7) circle (.3ex);
\draw[fill] (2,5) circle (.3ex);
\draw[fill] (5,5) circle (.3ex);
\draw[fill] (7,2) circle (.3ex);
\draw[fill] (8,5) circle (.3ex);

\draw[dashed](0,7)--(0,8)--(2,8);
\draw(2,8)--(8,8)--(8,2); 
\draw[dashed](8,2)--(8,0)--(7,0);
\draw (7,0)--(0,0)--(0,7);
\draw(0,4)--(2,4);
\draw[dashed](2,4)--(5,4)--(5,2);
\draw(5,2)--(5,0);
\draw[dashed] (0,7)--(2,7)--(2,5);
\draw(2,5)--(2,4);
\draw (2,8)--(2,7);
\draw[dashed] (2, 5)--(5, 5);
\draw (5,5)--(5,4);
\draw [dashed] (5,2)--(7,2)--(7,0);
\draw (7,2)--(8,2);
\draw (5, 5)--(8,5);

\end{tikzpicture}}

 \hspace{.4in} 

\adjustbox{valign=c}{
\begin{tikzpicture}[scale=.3]
\draw[fill] (0,8) circle (.3ex);
\draw[fill] (1,7) circle (.3ex);
\draw[fill] (2,6) circle (.3ex);
\draw[fill] (3,5) circle (.3ex);
\draw[fill] (4,4) circle (.3ex);
\draw[fill] (5,3) circle (.3ex);
\draw[fill] (6,2) circle (.3ex);
\draw[fill] (7,1) circle (.3ex);
\draw[fill] (8,0) circle (.3ex);
\node at (.5, 7.5) {$1$};
\node at (1.5, 6.5) {$2$};
\node at (2.5, 5.5) {$3$};
\node at (3.5, 4.5) {$4$};
\node at (4.5, 3.5) {$5$};
\node at (5.5, 2.5) {$6$};
\node at (6.5, 1.5) {$7$};
\node at (7.5, .5) {$8$};
\draw[fill] (0,7) circle (.3ex);
\draw[fill] (2,8) circle (.3ex);
\draw[fill] (8,8) circle (.3ex);
\draw[fill] (8,2) circle (.3ex);
\draw[fill] (7,0) circle (.3ex);
\draw[fill] (0,0) circle (.3ex);
\draw[fill] (0,4) circle (.3ex);
\draw[fill] (2,4) circle (.3ex);
\draw[fill] (5,4) circle (.3ex);
\draw[fill] (5,2) circle (.3ex);
\draw[fill] (5,0) circle (.3ex);
\draw[fill] (2,7) circle (.3ex);
\draw[fill] (2,5) circle (.3ex);
\draw[fill] (5,5) circle (.3ex);
\draw[fill] (7,2) circle (.3ex);
\draw[fill] (8,5) circle (.3ex);

\draw(0,7)--(0,8)--(2,8);
\draw(2,8)--(8,8)--(8,2); 
\draw(8,2)--(8,0)--(7,0);
\draw (7,0)--(0,0)--(0,7);
\draw(0,4)--(2,4);
\draw(2,4)--(5,4)--(5,2);
\draw(5,2)--(5,0);
\draw (0,7)--(2,7)--(2,5);
\draw(2,5)--(2,4);
\draw (2,8)--(2,7);
\draw (2, 5)--(5, 5);
\draw (5,5)--(5,4);
\draw (5,2)--(7,2)--(7,0);
\draw (7,2)--(8,2);
\draw (5, 5)--(8,5);

\draw[red](1,1)--(1, 4.5)--(1,7.5)--(7.5, 7.5)--(7.5,4.5)--(7.5, 1)--(6, 1)--(1,1);
\draw[red] (1, 4.5)--(3, 4.5)--(7.5, 4.5);
\draw[fill,red] (1,1) circle (.3ex);
\draw[fill,red] (1,4.5) circle (.3ex);
\draw[fill,red] (1,7.5) circle (.3ex);
\draw[fill,red] (7.5,7.5) circle (.3ex);
\draw[fill,red] (7.5, 4.5) circle (.3ex);
\draw[fill,red] (7.5, 1) circle (.3ex);
\draw[fill,red] (6,1) circle (.3ex);
\draw[fill,red] (1,1) circle (.3ex);
\draw[fill,red] (3,4.5) circle (.3ex);

\end{tikzpicture}}

\hspace{.4in}
\adjustbox{valign=c}{
\begin{tikzpicture}[scale=.5]
\draw[red](0,0)--(-1,1)--(0,2)--(1,3)--(0,4);
\draw[red](0,0)--(1,1)--(2,2)--(1,3);
\draw[red] (-1,1)--(-2, 2)--(0,4);
\draw[fill,red] (0,0) circle (.2ex); \node[red] at (0,-.4) {$5$};
\draw[fill,red] (-1,1) circle (.2ex); \node[red] at (-1.3,.8) {$2$};
\draw[fill,red] (0,2) circle (.2ex); \node[red] at (0.1,1.7) {$4$};
\draw[fill,red] (1,3) circle (.2ex); \node[red] at (1.2,3.3) {$6$};
\draw[fill,red] (0,4) circle (.2ex); \node[red] at (0,4.4) {$3$};
\draw[fill,red] (1,1) circle (.2ex); \node[red] at (1.2,.7) {$7$};
\draw[fill,red] (2,2) circle (.2ex); \node[red] at (2.3,2) {$8$};
\draw[fill,red] (-2,2) circle (.2ex); \node[red] at (-2.2,2) {$1$};
\end{tikzpicture}}

\end{tabular}\caption{The rectangulation and adjacency poset obtained from the twisted Baxter permutation $52147863$.
}\label{rect}
\end{figure}
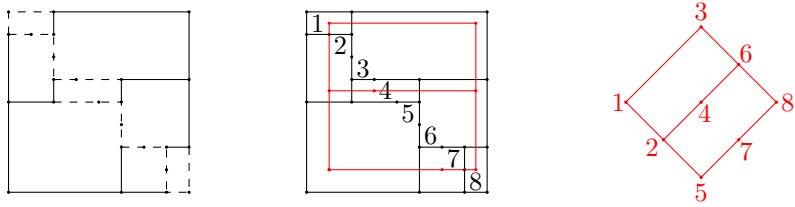

For our second result, which characterizes adjacency posets, we require the following definitions.
A poset $P$ is \emph{bounded} if it has an element that is greater than all other elements and an element that is less than all other elements.
Given a poset~$P$ on $[n]$, a \emph{2-14-3 chain} is a chain $b<_P a\lessdot_P d <_P c$ of $P$ such that $a<b<c<d$ in numerical order.
We similarly define a 3-14-2 chain, a 2-41-3 chain, and a 3-41-2 chain.

Given a partially ordered set $P$, construct a graph $G$ such that the vertices of~$G$ are labeled by the elements of $P$ and there is an edge joining vertex $x$ to vertex~$y$ if and only if $x\lessdot_P y$ or $y\lessdot_P x$.
An embedding of $G$ in $\mathbb{R}^2$ is a \emph{Hasse diagram} for~$P$ if and only if for all $x\lessdot_P y$, vertex $y$ is above vertex $x$ in the plane and each edge of the embedding is a line segment.
A \emph{planar embedding} of a poset $P$ is a Hasse diagram for $P$ in which no two edges intersect.

\begin{definition}\label{def:baxposet}
A poset $P$ on $[n]$ is a \emph{Baxter poset} if and only if it satisfies the following conditions:
\begin{enumerate}
\item $P$ is bounded.
\item If $x\in P$, then
$x$ is covered by at most two elements and covers at most two elements. 
\item $P$ contains no $2$-$14$-$3$, no $3$-$14$-$2,$ no $2$-$41$-$3,$ and no $3$-$41$-$2$ chains.
\item If $[x,y]$ is an interval of $P$ such that the open interval $(x,y)$ is disconnected, then
 $|x-y|=1$. 
\item There exists a planar embedding of $P$ such that for every interval $[x,y]$ of~$P$ with $(x,y)$ disconnected, if $w,z \in (x,y)$  and $w$ is left of~$z$, then $w<x<z$ in numerical order.  
\end{enumerate}
\end{definition}

Condition 2 of Definition \ref{def:baxposet} implies that every open interval $(x,y)$ of a Baxter poset consists of at most two connected components.
Condition 5 implies that if $(x,y)$ is disconnected, then the elements of one connected component are all smaller than $x$ and $y$ while the elements of the other connected component are larger than~$x$ and $y$.
We call the embedding described in Condition 5 of Definition \ref{def:baxposet} a \emph{natural embedding} of the Baxter poset. 

We can now state our main result. 
\begin{theorem}\label{thm:adjposet}
A poset $P$ is a Baxter poset if and only if it is the adjacency poset of a diagonal rectangulation.
\end{theorem}

\begin{remark}
One might hope for an unlabeled version of the Baxter poset from which the labeled poset can be obtained, just as sub-binary trees have a canonical labeling.
However, without ``decorating" the poset with additional combinatorial information, this is not possible.
This is quickly apparent since, when $n=4$, of the 22 Baxter posets, 20 of these are chains.
Decorating each poset to indicate the numerical order of each pair $x<_P y$ with $(x,y)$ disconnected is insufficient.
Additionally, decorating every edge of the Hasse diagram to indicate the numerical order of the elements of the cover relation does not allow us to determine a unique Baxter poset.
\end{remark}

The original Baxter object, Baxter permutations, have a pattern avoidance definition similar to the definition of the twisted Baxter permutations.
A \emph{Baxter permutation} $\sigma=\sigma_1 \cdots \sigma_n$ is a permutation that avoids the patterns 2-41-3 and 3-14-2, i.e. there does not exist $i<j<j+1<k$ such that $\sigma_{j+1}<\sigma_i< \sigma_k<\sigma_j$ or $\sigma_j<\sigma_k<\sigma_i<\sigma_{j+1}$.
Given a diagonal rectangulation $D$, the set of permutations that map to $D$ under the binary tree gluing map contains a unique twisted Baxter permutation and a unique Baxter permutation (see Theorem \ref{thm:rho}). 
Other authors (see \cite[Proof of Lemma 8.4]{drec}, \cite[Proof of Lemma 6.6]{felsner}) have described algorithms for obtaining these permutations from a diagonal rectangulation.
Our final results describe how to obtain these pattern avoiding permutations directly from a Baxter poset. 
Here, we describe a method of obtaining the Baxter permutation.

Let $P$ be the natural embedding of a Baxter poset. 
The edges of the embedding separate the plane into maximal connected components.
We call the closure of a bounded connected component a \emph{region} of the embedding.
Assign an arrow to each region of the embedding as follows:  If the maximal element of a region is greater (in numerical order) than the minimal element of that region, then that region is assigned a right-pointing arrow, and otherwise the region is assigned a left-pointing arrow.
An example is shown in Figure \ref{baxperm}.
If a region $R_i$ contains a right-pointing arrow and $\sigma$ is a linear extension of $P$ in which all labels of elements contained in the left side of $R_i$ precede all labels of elements contained in the right side of~$R_i$, then we say that $\sigma$ \emph{respects the arrow of $R_i$}. 
Similarly, we say that~$\sigma$ respects the arrow of a region~$R_i$ containing a left-pointing arrow if all labels contained in the right side of $R_i$ precede all labels of elements contained in the left side of $R_i$.  
If~$\sigma$ respects the arrows of every region of~$P$, then we say that $\sigma$ \emph{respects the arrows of~$P$}.
The existence of a linear extension of $P$ that respects the arrows of $P$ should not be immediately obvious to the reader.

\begin{theorem}\label{thm:baxter}
Given a Baxter poset $P$ with its natural embedding, 
the unique Baxter permutation that is a linear extension of $P$ is the unique linear extension that respects the arrows of the embedding.
\end{theorem}

By adding a single relation for each region of the natural embedding of $P$, we obtain an alternate description of the map from an adjacency poset to its Baxter permutation.
Specifically, for each region $R$ with minimal element $x$ and maximal element $x+1$ we declare that the maximal element (with respect to the partial order~$P$) of the left component of $(x,y)$ is less than the minimal element of the right component. 
Similarly, for each region $R$ with maximal element $x$ and minimal element $x+1$, we declare that the maximal element of the right component of $(x,y)$ is less than the minimal element of the left component.
By Theorem \ref{thm:baxter}, the resulting partial order is a total order on $[n]$ and this total order is a Baxter permutation.

\begin{figure}
\begin{center}
\begin{tikzpicture}[scale=.3]
\draw (0,0)--(0,12)--(12,12)--(12,0)--(0,0);
\draw (0,2)--(12,2);
\draw (11, 0)--(11,2);
\draw (7, 2)--(7,12);
\draw (9, 2)--(9, 12);
\draw (7, 4)--(9,4);
\draw (0,8)--(6,8)--(6,2);
\draw(5,2)--(5,8);
\draw (1,8)--(1,12);
\draw (1, 9)--(7,9);
\draw (2, 9)--(2,12);
\draw (6, 8)--(6,9);
\node at (.5,11.5) {\footnotesize $1$};
\node at (1.5, 10.5) {\footnotesize $2$};
\node at (2.5, 9.5) {\footnotesize $3$};
\node at (3.5, 8.5) {\footnotesize $4$};
\node at (4.5, 7.5) {\footnotesize $5$};
\node at (5.5, 6.5) {\footnotesize $6$};
\node at (6.5, 5.5) {\footnotesize $7$};
\node at (7.5, 4.5) {\footnotesize $8$};
\node at (8.5, 3.5) {\footnotesize $9$};
\node at (9.5, 2.5) {\footnotesize $10$};
\node at (10.5, 1.5) {\footnotesize $11$};
\node at (11.5, .5) {\footnotesize $12$};

\draw [red] (23, 1)--(18+5/4*2, 1+5/4*2)--(18+5/4*3, 1+5/4*3)--(18+5/4*2, 1+5/4*4)--(18+5/4*3, 1+5/4*5)--(18+5/4*4, 1+5/4*6)--(18+5/4*3, 1+5/4*7)--(18+5/4*4, 1+5/4*8);
\draw [red] (18+5/4*2, 1+5/4*2)--(18+5/4*1, 1+5/4*3)--(18+5/4*2, 1+5/4*4);
\draw[red] (18+5/4*2, 1+5/4*4)--(18+5/4*1, 1+5/4*5)--(18+5/4*2, 1+5/4*6);
\draw[red] (18+5/4*3, 1+5/4*5)--(18+5/4*2, 1+5/4*6)--(18+5/4*3, 1+5/4*7);
\draw[red] (23,1)--(28,6)--(23,11);
\draw[fill, red] (23,1) circle (.3ex); \node[red] at (23,1-.5) {\footnotesize $11$};
\draw[fill, red] (18+5/4*2, 1+5/4*2) circle (.3ex); \node[red] at (18+5/4*2-.3, 1+5/4*2-.3) {\footnotesize $5$};
\draw[fill, red] (18+5/4*3, 1+5/4*3) circle (.3ex); \node[red] at (18+5/4*3+.4, 1+5/4*3) {\footnotesize $6$};
\draw[fill, red] (18+5/4*2, 1+5/4*4) circle (.3ex);\node[red] at (18+5/4*2-.6, 1+5/4*4) {\footnotesize $4$};
\draw[fill, red] (18+5/4*3, 1+5/4*5) circle (.3ex); \node[red] at (18+5/4*3+.35, 1+5/4*5-.3) {\footnotesize $7$};
\draw[fill, red] (18+5/4*4, 1+5/4*6) circle (.3ex);  \node[red] at (18+5/4*4+.4, 1+5/4*6) {\footnotesize $9$};
\draw[fill, red] (18+5/4*3, 1+5/4*7) circle (.3ex); \node[red] at (18+5/4*3-.3, 1+5/4*7+.3) {\footnotesize $8$};
\draw[fill, red] (18+5/4*4, 1+5/4*8) circle (.3ex);  \node[red] at (18+5/4*4, 1+5/4*8+.5) {\footnotesize $10$};
\draw[fill, red] (18+5/4*1, 1+5/4*3) circle (.3ex);  \node[red] at (18+5/4*1-.4, 1+5/4*3) {\footnotesize $1$};
\draw[fill, red] (18+5/4*1, 1+5/4*5) circle (.3ex);  \node[red] at (18+5/4*1-.4, 1+5/4*5) {\footnotesize $2$};
\draw[fill, red] (18+5/4*2, 1+5/4*6) circle (.3ex);  \node[red] at (18+5/4*2-.3, 1+5/4*6+.3) {\footnotesize $3$};
\draw[fill, red] (28,6) circle (.3ex);  \node[red] at (28+.5,6) {\footnotesize $12$};
\draw[<-, thick] (18+5/4*3.5, 1+5/4*6)--(18+5/4*2.5, 1+5/4*6);
\draw[<-, thick] (18+5/4*1.5, 1+5/4*3)--(18+5/4*2.5, 1+5/4*3);
\draw[<-, thick] (18+5/4*1.5, 1+5/4*5)--(18+5/4*2.5, 1+5/4*5);
\draw[<-, thick] (18+5/4*4.5, 6)--(18+5/4*5.5, 6);

\draw (17, -1) circle (.65cm);
\draw (18.6, -1) circle (.65cm);
\draw (29, -1) circle (.65cm);

\node at (23, -1) {11  {12}  5 {6} 1 {4} 7 {2} 3 {9} 8 {10}};
\end{tikzpicture}
\end{center}\caption{A diagonal rectangulation, the corresponding adjacency poset with an arrow assigned to each region, and the Baxter permutation obtained using Theorem \ref{thm:baxter}}\label{baxperm}
\end{figure}
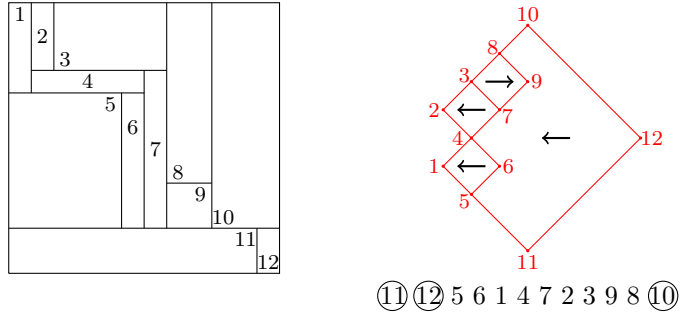

In Section \ref{definitions}, we describe the map $\rho$ from permutations to diagonal rectangulations that coincides with the binary tree gluing map already described and provide some background related to diagonal rectangulations.
We prove Theorem \ref{thm:adjcovers} (the characterization of the cover relations of the adjacency poset) in Section~\ref{adjposet}.
Our main result, Theorem~\ref{thm:adjposet}, is proved in Section \ref{bax}.
Finally, in Section \ref{perms}, we describe how to obtain a twisted Baxter permutation from a Baxter poset and then prove Theorem \ref{thm:baxter}.

\section{Diagonal Rectangulations}\label{definitions}

A \emph{rectangulation of size n} is an equivalence class of decompositions of a square~$S$ into $n$ rectangles. 
Two decompositions $R_1$ and $R_2$ are members of the same equivalence class if and only if there exists a homeomorphism of the square, fixing its vertices, that takes $R_1$ to $R_2$. 
We say that a rectangulation is a \emph{diagonal rectangulation} if, for some representative of the equivalence class, the top-left to bottom-right diagonal of $S$ contains an interior point of each rectangle of the decomposition. 
In our discussion of diagonal rectangulations, we often blur the distinction between an equivalence class and a representative of the equivalence class.
We most often refer to a diagonal rectangulation using the distinguished representative with edges intersecting the diagonal in equally spaced points.

We now define a map $\rho$ from $S_n$ to the set of diagonal rectangulation of size~$n$.  
Figure~\ref{fig:rho} shows the construction of $\rho(23154)$.
The map $\rho$ agrees with the map (described in Section \ref{intro}) in which a diagonal rectangulation is constructed from a permutation by gluing together twin binary trees and then rotating the result.
Our description of $\rho$ matches the description in \cite[Section 6]{drec} and is essentially equivalent to maps described in  \cite[Section 3]{Ackerman2}, \cite[Section 4]{Ackerman}, and \cite[Section 5]{felsner}.

Let $\sigma=\sigma_1 \cdots \sigma_n \in S_n$ and $S$ a square in $\mathbb{R}^2$ with bottom-left vertex at~$(0,0)$ and top-right vertex at $(n,n)$.
Place $n+1$ points at $(i,  n-i)$ for $i\in \{0,\ldots,n\}$.
Label each of the~$n$ spaces between these points in order with an element of $[n]$, starting with~1 in the upper-leftmost space and finishing with $n$ in the lower-rightmost space.
We construct~$\rho(\sigma)$ by considering the entries of $\sigma$ sequentially from left to right.
Let~$T_{i-1}$ denote the union of the left and lower boundaries of $S$ and the rectangles of $\rho(\sigma)$ constructed using the first $i-1$ entries of $\sigma$.  
In step $i$ of the construction, we form a new rectangle that contains the diagonal label $\sigma_i$.
We refer to this rectangle as \emph{rectangle $\sigma_i$}.
We construct rectangle $\sigma_i$ as follows.
If the point $u=(\sigma_i-1, n-(\sigma_i-1))$ is contained in~$T_{i-1}$, then place the upper-left corner of rectangle $\sigma_i$ so that it coincides with the uppermost point on the segment of~$T_{i-1}$ containing~$u$.
Otherwise, the upper-left corner of rectangle $\sigma_i$ is the first point of $T_{i-1}$ hit by the left-pointing horizontal ray with base point at $u$.
If the point $l=(\sigma_i, n-\sigma_i)$ is contained in $T_{i-1}$, then place the lower-right corner of rectangle~$\sigma_i$ so that it coincides with the rightmost point on the segment of $T_{i-1}$ containing~$l$.
Otherwise, the lower-right corner of rectangle $\sigma_i$ is the first point of $T_{i-1}$ hit by the downward pointing vertical ray with base point at~$l$.
In the arguments that follow, we will use the observation that, by construction, the left side and bottom of rectangle $\sigma_i$ are contained in~$T_{i-1}$ for all $i \in [n]$.
We will also use the observation that, since the interior of each rectangle of a diagonal rectangulation $D$ intersects the upper-left to bottom-right diagonal of $S$, no set of four rectangles of $D$ share a vertex.

\begin{theorem}[{\cite[Theorem 6.1, Corollary 8.7]{drec}}] \label{thm:rho}
The map $\rho$ restricts to a bijection between twisted Baxter permutations and diagonal rectangulations. 
The map $\rho$ also restricts to a bijection between Baxter permutations and diagonal rectangulations.
\end{theorem}

Given a rectangulation $R$, a line segment that is not contained in the boundary of $S$ and is a maximal (with respect to inclusion) union of edges of rectangles is called a \emph{wall} of $R$.

Recall that a permutation $\sigma$ is a twisted Baxter permutation if and only if it avoids the patterns 2-41-3 and 3-41-2.
This pattern avoidance is equivalent to the requirement that if $\sigma_i>\sigma_{i+1}$ then either all values numerically between~$\sigma_{i+1}$ and~$\sigma_i$ are left of~$\sigma_i$ in $\sigma$, or all of these values are right of~$\sigma_{i+1}$ in $\sigma$.

We say that two permutations $\sigma$ and $\psi$ are related by a ($3$-$14$-$2 \leftrightarrow 3$-$41$-$2$) \emph{move} if~$\sigma$ contains a subsequence $\sigma_{i_1} \sigma_{i_2} \sigma_{i_3} \sigma_{i_4}$ that is an occurrence of one of these patterns and switching the positions of the adjacent entries $\sigma_{i_2}$ and $\sigma_{i_3}$ in $\sigma$ results in the permutation $\psi$.  
We say that~$\sigma$ and $\psi$ are related by a \emph{$(2$-$14$-$3 \leftrightarrow 2$-$41$-$3)$ move} if $\sigma$ and $\psi$ satisfy the same conditions with these patterns. 

\begin{proposition}[{\cite[Proposition 6.3]{drec}}]\label{prop:equivclass}
Two permutations $\sigma$ and $\psi$ satisfy $\rho(\sigma)=\rho(\psi)$ if and only if they are related by a sequence of $(3$-$14$-$2 \leftrightarrow 3$-$41$-$2)$ moves and  ${(2\text{-}14\text{-}3} \leftrightarrow {2\text{-}41\text{-}3})$ moves. 
\end{proposition}

Given $\psi\in S_n$, define $\text{inv}(\psi)=\{(\psi_i, \psi_j) \ | \ i<j \text{ and } \psi_i>\psi_j\}$.
If $\sigma, \psi \in S_n$  then we say that $\sigma\leq\psi$ in the \emph{right weak order} if and only if $\text{inv}(\sigma)\subseteq \text{inv}(\psi)$. 
This definition implies that $\sigma \lessdot \psi$ in the right weak order if and only if $\psi$ can be obtained from $\sigma$ by transposing adjacent entries $\sigma_i$ and $\sigma_{i+1}$ of $\sigma$ which satisfy $\sigma_i<\sigma_{i+1}$ in numerical order.

\begin{proposition}[{\cite[Proposition 4.5]{drec}}]\label{prop:minmax}
Let $D$ be a diagonal rectangulation and $\sigma\in S_n$ such that $\rho(\sigma)=D$.
Then $\sigma$ is a twisted Baxter permutation if and only if~$\sigma$ is the minimal element of the right weak order such that $\rho(\sigma)=D$.  
\end{proposition}

\begin{figure}
\begin{tikzpicture}[scale=.35]
\draw(0,0) rectangle (5,5);
\draw[fill=light-gray] (0,0) rectangle (2,4);
\draw[fill] (0,5) circle (.5ex);
\draw[fill] (1,4) circle (.5ex);
\draw[fill] (2,3) circle (.5ex);
\draw[fill] (3,2) circle (.5ex);
\draw[fill] (4,1) circle (.5ex);
\draw[fill] (5,0) circle (.5ex);
\node at (1.5, 3.5) {2};
\node at (2.5, 2.5) {3};
\node at (3.5, 1.5) {4};
\node at (4.5,.5) {5};
\node at (.5, 4.5) {1};

\draw(7,0) rectangle (12,5);
\draw[fill=mgray] (7,0) rectangle (9,4);
\draw[fill=light-gray] (9,0) rectangle (10,4);
\draw[fill] (7,5) circle (.5ex);
\draw[fill] (8,4) circle (.5ex);
\draw[fill] (9,3) circle (.5ex);
\draw[fill] (10,2) circle (.5ex);
\draw[fill] (11,1) circle (.5ex);
\draw[fill] (12,0) circle (.5ex);
\node at (8.5, 3.5) {2};
\node at (9.5, 2.5) {3};
\node at (10.5, 1.5) {4};
\node at (11.5,.5) {5};
\node at (7.5, 4.5) {1};

\draw(14,0) rectangle (19,5);
\draw[fill=mgray] (14,0) rectangle (16,4);
\draw[fill=mgray] (16,0) rectangle (17,4);
\draw[fill=light-gray] (14,4) rectangle (17, 5);
\draw[fill] (14,5) circle (.5ex);
\draw[fill] (15,4) circle (.5ex);
\draw[fill] (16,3) circle (.5ex);
\draw[fill] (17,2) circle (.5ex);
\draw[fill] (18,1) circle (.5ex);
\draw[fill] (19,0) circle (.5ex);
\node at (15.5, 3.5) {2};
\node at (16.5, 2.5) {3};
\node at (17.5, 1.5) {4};
\node at (18.5,.5) {5};
\node at (14.5, 4.5) {1};

\draw(21,0) rectangle (26,5);
\draw[fill=mgray] (21,0) rectangle (23,4);
\draw[fill=mgray] (23,0) rectangle (24,4);
\draw[fill=mgray] (21,4) rectangle (24, 5);
\draw[fill=light-gray](24, 0)rectangle (26, 1);
\draw[fill] (21,5) circle (.5ex);
\draw[fill] (22,4) circle (.5ex);
\draw[fill] (23,3) circle (.5ex);
\draw[fill] (24,2) circle (.5ex);
\draw[fill] (25,1) circle (.5ex);
\draw[fill] (26,0) circle (.5ex);
\node at (22.5, 3.5) {2};
\node at (23.5, 2.5) {3};
\node at (24.5, 1.5) {4};
\node at (25.5,.5) {5};
\node at (21.5, 4.5) {1};

\draw(28,0) rectangle (33,5);
\draw[fill=mgray] (28,0) rectangle (30,4);
\draw[fill=mgray] (30,0) rectangle (31,4);
\draw[fill=mgray] (28,4) rectangle (31, 5);
\draw[fill=mgray](31, 0)rectangle (33, 1);
\draw[fill=light-gray](31,1) rectangle (33, 5);
\draw[fill] (28,5) circle (.5ex);
\draw[fill] (29,4) circle (.5ex);
\draw[fill] (30,3) circle (.5ex);
\draw[fill] (31,2) circle (.5ex);
\draw[fill] (32,1) circle (.5ex);
\draw[fill] (33,0) circle (.5ex);
\node at (29.5, 3.5) {2};
\node at (30.5, 2.5) {3};
\node at (31.5, 1.5) {4};
\node at (32.5,.5) {5};
\node at (28.5, 4.5) {1};
\end{tikzpicture}
\caption{The map $\rho$ is applied to the permutation $23154$.}
\label{fig:rho}
\end{figure}
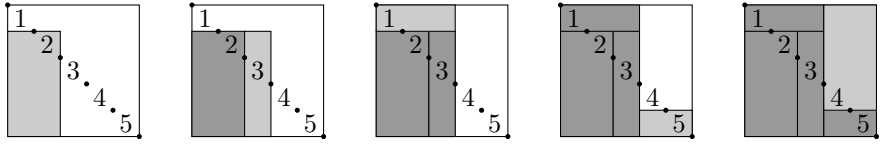

\section{The adjacency poset of a diagonal rectangulation}\label{adjposet}

In Section \ref{intro}, we provided a definition of the adjacency poset of a diagonal rectangulation~$D$.
At times, we will make use of an equivalent definition.

Given a diagonal rectangulation $D$ of size $n$ in $\mathbb{R}^2$ with bottom-left corner at $(0,0)$ and top-right corner at $(n,n)$, define the partial order $
Q$ on~$[n]$ as follows:  if there exist a point $p$ in the interior of rectangle $x$ and a point~$q$ in the interior of rectangle $y$ such that $q-p$ has positive coordinates declare $x\leq_Q y$, and then take the transitive closure of these relations.

\begin{proposition}
Given a diagonal rectangulation $D$ of size $n$, the adjacency poset~$P$ is the poset $Q$ defined above.
\end{proposition}

\begin{proof}
If $x\lessdot_P y$ then, by the definition of the adjacency poset, the interior of the bottom (or left side) of rectangle~$y$ intersects the interior of the top (or right side) of rectangle~$x$.
Thus there exist points $p\in \text{int(rectangle } x)$ and $q \in \text{int(rectangle } y)$ such that $q-p$ has positive coordinates.
Therefore, by the definition of $Q$, we have that $x\leq_Q y$.

If $x\lessdot_Qy$, then there exist points $p\in \text{int(rectangle } x)$ and $q \in \text{int(rectangle } y)$ such that $q-p$ has positive coordinates.
Consider the line segment joining $p$ to~$q$.
If this segment passes through the vertex of some rectangle, since $D$ contains only finitely many vertices, we may perturb $p$ or $q$, obtaining points $p'$ and $q'$, so that $p'$ and $q'$ are respectively in the interiors of rectangles $x$ and $y$, the segment joining $p'$ and $q'$ contains no vertices of $D$, and $q'-p'$ has positive coordinates.
Thus, we may assume that the segment joining $p$ and $q$ contains no vertices of $D$.
The segment passes through the interiors of the sequence of rectangles $x=z_0, z_1, \ldots,z_{m-1}, y=z_m$.
For all $i \in [m]$, the segment exits rectangle~$z_{i-1}$ and enters rectangle $z_{i}$ at a point in the interior of a side of both rectangles so $z_i<_P z_{i+1}$.
Therefore $x<_P y$.
\end{proof}

We note that the transitive closure in the definition of $Q$ is required (since we have chosen to refer to each diagonal rectangulation using the representative with edges intersecting the diagonal in equally spaced points).
Consider the rectangulation $\rho(312465)$ shown in Figure \ref{fig:natembed}.
Since the interior of the right side of rectangle 2 intersects the interior of the left side of rectangle 4, we have that $2<_P 4$.
Similarly, $4<_P 6$, so by transitivity $2<_P 6$.
However, there do not exist $p \in \text{int}(\text{rectangle 2})$ and $q \in \text{int}(\text{rectangle 6})$ such that $q-p$ has positive coordinates.

We give a description of the Hasse diagram of the adjacency poset of a diagonal rectangulation by describing its cover relations.

\begin{theorem}\label{thm:adjcovers}
Let $D$ be a diagonal rectangulation and $P$ the corresponding adjacency poset.  
Then $x\lessdot_P y$ if and only if rectangles $x$ and $y$ form one of the configurations shown in Figure \ref{fig:adjcovers}.
 \end{theorem}

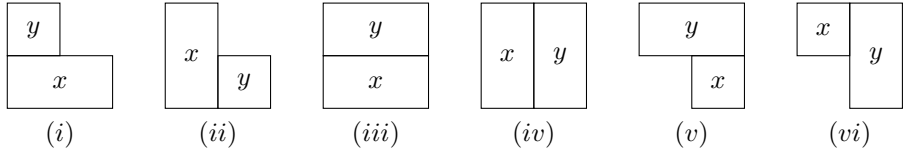
\begin{figure}
\begin{tikzpicture}[scale=.7]
\draw (0,0) rectangle (2,1);
\draw (0,1) rectangle (1,2);
\node at (1,.5)   {$x$} ;
\node at (.5,1.5) {$y$};
\node at (1, -.5) {$(i)$};

\draw (3,0) rectangle (4,2);
\draw (4,0) rectangle (5,1);
\node at (3.5,1)   {$x$} ;
\node at (4.5,.5) {$y$};
\node at (4, -.5) {$(ii)$};

\draw (6,0) rectangle (8,1);
\draw (6,1) rectangle (8,2);
\node at (7,.5)   {$x$} ;
\node at (7,1.5) {$y$};
\node at (7, -.5) {$(iii)$};

\draw (9,0) rectangle (10,2);
\draw (10,0) rectangle (11,2);
\node at (9.5,1)   {$x$} ;
\node at (10.5,1) {$y$};
\node at (10, -.5) {$(iv)$};

\draw (12,1) rectangle (14, 2);
\draw(13,0) rectangle (14,1);
\node at (13.5,.5)   {$x$} ;
\node at (13,1.5) {$y$};
\node at (13, -.5) {$(v)$};

\draw(15,1) rectangle (16,2);
\draw (16,0) rectangle (17,2);
\node at (15.5,1.5)   {$x$} ;
\node at (16.5,1) {$y$};
\node at (16, -.5) {$(vi)$};
\end{tikzpicture}
\caption{Configurations in a diagonal rectangulation that correspond to cover relations in the adjacency poset.}
\label{fig:adjcovers}
\end{figure}

\begin{proof}
Let $D$ be a diagonal rectangulation and $P$ the adjacency poset of $D$.
Assume that in $D$, rectangles $x$ and $y$ form one of the configurations shown in Figure \ref{fig:adjcovers}.
In each configuration, by definition, $x<_P y$.
Assume that rectangles $x$ and $y$ form configuration ($i$) and there exists some $z\in [n]$ such that $x<_P z<_P y$.
Since $z<_P y$ and $P$ is acyclic, $y\nless_P z$.  
Thus rectangle $z$ contains no interior points in the lined region of Figure \ref{fig:adji}.
Similarly, since $z \nless_P x$, rectangle $z$ contains no interior points in the dotted region of Figure \ref{fig:adji}. 
Therefore, any rectangle $z$ such that $x<_P z <_P y$ is completely contained in an unshaded region of Figure \ref{fig:adji}.
However, by the definition of $P$, no label of a rectangle contained in the lower-right unshaded region of Figure \ref{fig:adji} is covered by $y$. 
Similarly, in $P$ no label of a rectangle contained in the upper-left unshaded region of Figure \ref{fig:adji} covers $x$.
Additionally, no label of a rectangle contained in the lower-right unshaded region is covered by the label of a rectangle contained in the upper-left unshaded region.   
Thus there exists no~$z$ such that $x<_P z <_P y$.
Hence $x\lessdot_P y$.
For the remaining configurations of Figure \ref{fig:adjcovers}, similar considerations demonstrate that $x\lessdot_P y$.

\begin{figure}
\begin{tikzpicture}[scale=.6]
\draw[draw=none,pattern=north west lines, pattern color=gray] (0,1) rectangle (3,3);
\draw[draw=none,pattern=crosshatch dots, pattern color=gray] (-1,-1) rectangle (2,1);
\draw (-1,-1) rectangle (3,3);
\draw (0,0) rectangle (2,1);
\draw (0,1) rectangle (1,2);
\node at (1,.5)   {$x$} ;
\node at (.5,1.5) {$y$};

\end{tikzpicture} 
\caption{An illustration for the proof of Theorem \ref{thm:adjcovers}.}
\label{fig:adji}
\end{figure}
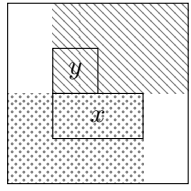

To prove the other direction of the theorem, assume that $x\lessdot_P y$.  
Since the set of linear extensions of $P$ is the fiber $\rho^{-1}(D)$ and $x\lessdot_P y$, there exists a linear extension $\sigma=\sigma_1 \cdots \sigma_n$ of $P$ such that $x=\sigma_i$ and $y=\sigma_{i+1}$.
Let $T_{j-1}$ denote the union of the left and bottom boundaries of the square $S$ and the partial diagonal rectangulation formed in the construction of $\rho(\sigma)$ after considering the first $j-1$ entries of $\sigma$.
The bottom and left edge of rectangle $\sigma_j$ is contained in $T_{j-1}$ for all $j\in [n]$.
Using the definition of the adjacency poset from Section \ref{intro}, since $x\lessdot_P y$, we have that rectangles $x$ and $y$ are adjacent with rectangle $x$ left of or below rectangle $y$.
Thus, combining these requirements, 
rectangles $x$ and $y$ form one of the configurations shown in Figure \ref{fig:config}.

\begin{figure}
\begin{tikzpicture}[scale=.6]

\draw (0,0) rectangle (1,2);
\draw (1,0) rectangle (2,1);
\node at (.5,1) {$x$};
\node at (1.5, .5) {$y$};
\node at (1,-.75) {$(e)$};

\draw (0,5) rectangle (1,6.5);
\draw (1,4.5) rectangle (2,6);
\node at (.5, 5.75) {$x$};
\node at (1.5, 5.25) {$y$};
\node at (1,3.75) {$(a)$};

\draw (4, 0) rectangle (5,2);
\draw (5,0) rectangle (6,2);
\node at (4.5,1) {$x$};
\node at (5.5, 1) {$y$};
\node at (5,-.75) {$(f)$};

\draw (4, 5.5) rectangle (5,6.5);
\draw (5,4.5) rectangle (6,6.5);
\node at (4.5, 6) {$x$};
\node at (5.5, 5.5) {$y$};
\node at (5,3.75) {$(b)$};

\draw (8,0) rectangle (10,1);
\draw (8,1) rectangle (9,2);
\node at (9, .5) {$x$};
\node at (8.5, 1.5) {$y$};
\node at (9,-.75) {$(g)$};

\draw (8.5, 4.5) rectangle (10, 5.5);
\draw (8,5.5) rectangle (9.5, 6.5);
\node at (9.25, 5) {$x$};
\node at (8.75, 6) {$y$};
\node at (9,3.75) {$(c)$};

\draw (12,0) rectangle (14, 1);
\draw (12,1) rectangle (14, 2);
\node at (13, .5) {$x$};
\node at (13, 1.5) {$y$};
\node at (13,-.75) {$(h)$};

\draw (13, 4.5) rectangle (14, 5.5);
\draw (12, 5.5) rectangle (14, 6.5);
\node at (13.5, 5) {$x$};
\node at (13, 6) {$y$};
\node at (13,3.75) {$(d)$};

\end{tikzpicture}
\caption{Relative locations of rectangles $x$ and $y$ used in the second half of the proof of 
Theorem \ref{thm:adjcovers}.}
\label{fig:config}
\end{figure}
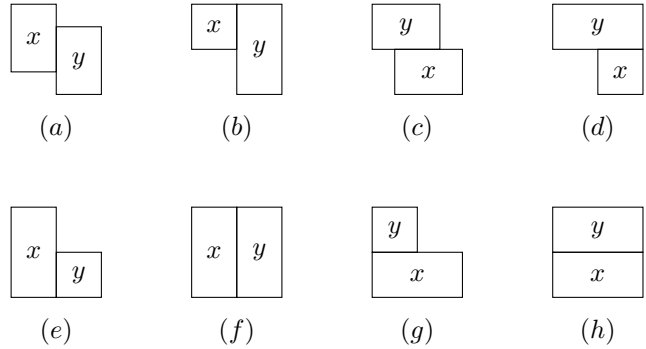  

To complete the proof of the theorem, we observe that configurations (a) and (c) of Figure~\ref{fig:config} cannot occur in any diagonal rectangulation.   
In a \emph{diagonal} rectangulation, the upper-left to bottom-right diagonal of $S$ passes through every rectangle of the rectangulation, but this is impossible in a rectangulation containing either of these configurations.  
Thus, if $x \lessdot_P y$, then rectangles $x$ and $y$ form one of the configurations shown in Figure \ref{fig:adjcovers}.
\end{proof}

\begin{figure}
\begin{tikzpicture}[scale=.4]
\draw (0,0) rectangle (6,6);
\node at (.5, 5.5) {1};
\node at (1.5, 4.5) {2};
\node at (2.5, 3.5) {3};
\node at (3.5, 2.5) {4};
\node at (4.5, 1.5) {5};
\node at (5.5,.5) {6};
\draw (0,4) --(3,4) -- (3,0);
\draw (1,4) -- (1,6);
\draw (3,2) -- (5,2)--(5,0);
\draw (3,4) -- (3, 6);
\draw (5, 2) --(5,6);

\draw[red] (.5,.5)--(.5,5)--(5.5, 5);
\draw[red] (.5, .5)--(4,.5)--(4, 5);

\draw[red,fill=red] (.5,.5) circle (.5ex);
\draw[red,fill=red] (.5,5) circle (.5ex);
\draw[red,fill=red] (2,5) circle (.5ex);
\draw[red,fill=red] (4,5) circle (.5ex);
\draw[red,fill=red] (4,.5) circle (.5ex);
\draw[red,fill=red] (5.5,5) circle (.5ex);

\node at (3,-1) {$\rho(312546)$};

\draw[red] (12,0.5)--(10,1.5)--(10,3.5)--(12,4.5)--(12,5.5);
\draw[red] (12,0.5)--(14,2.5)--(12,4.5);
\draw[red,fill=red] (12,0.5) circle (.5ex);
\draw[red,fill=red] (10,1.5) circle (.5ex);
\draw[red,fill=red] (10,3.5) circle (.5ex);
\draw[red,fill=red] (12,4.5) circle (.5ex);
\draw[red,fill=red] (12,5.5) circle (.5ex);
\draw[red,fill=red] (14,2.5) circle (.5ex);

\node at (12,0) {3};
\node at (9.5,1.5) {1};
\node at (9.5,3.5) {2};
\node at (12.5, 4.5) {4};
\node at (11.5, 5.5) {6};
\node at (14.5,2.5) {5};

\draw (0,-10) rectangle (6,-4);
\node at (.5, -4.5) {1};
\node at (1.5, -5.5) {2};
\node at (2.5, -6.5) {3};
\node at (3.5, -7.5) {4};
\node at (4.5, -8.5) {5};
\node at (5.5,-9.5) {6};
\draw (0,-6)--(3,-6)--(3, -10);
\draw (1,-6)--(1,-4);
\draw (3,-6)--(3,-4);
\draw (4, -10)--(4, -4);
\draw (4,-9)--(6,-9);

\node at (3, -11) {$\rho(312465)$};

\draw[red](.5,-9.5)--(.5,-5)--(2,-5)--(3.5,-5);
\draw[red](3.5, -9.5)--(5,-9.5)--(5,-8);

\draw[red,fill=red] (.5,-9.5) circle (.5ex);
\draw[red,fill=red] (.5,-5) circle (.5ex);
\draw[red,fill=red] (2,-5) circle (.5ex);
\draw[red,fill=red] (3.5,-5) circle (.5ex);
\draw[red,fill=red] (3.5,-9.5) circle (.5ex);
\draw[red,fill=red] (5,-9.5) circle (.5ex);
\draw[red,fill=red] (5,-8) circle (.5ex);

\draw[red] (12,-9.5)--(12,-4.5);

\draw[red,fill=red] (12,-9.5) circle (.5ex);
\draw[red,fill=red] (12,-8.5) circle (.5ex);
\draw[red,fill=red] (12,-7.5) circle (.5ex);
\draw[red,fill=red] (12,-6.5) circle (.5ex);
\draw[red,fill=red] (12,-5.5) circle (.5ex);
\draw[red,fill=red] (12,-4.5) circle (.5ex);

\node at (12.5,-9.5) {3};
\node at (12.5,-8.5) {1};
\node at (12.5,-7.5) {2};
\node at (12.5, -6.5) {4};
\node at (12.5, -5.5) {6};
\node at (12.5,-4.5) {5};
\end{tikzpicture}
\caption{The adjacency posets for the diagonal rectangulations $\rho(312546)$  and $\rho(312465)$.}
\label{fig:natembed}
\end{figure}
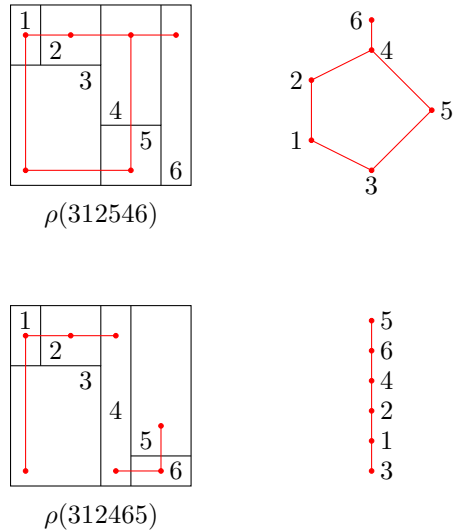

Figure \ref{fig:natembed} shows two diagonal rectangulations and their adjacency posets.
The posets are constructed using the correspondence between cover relations of $P$ and the rectangle configurations shown in Figure \ref{fig:adjcovers}.

\section{Characterization of Adjacency Posets}\label{bax}
To prove Theorem \ref{thm:adjposet}, we require the following definitions and results.
Given a planar embedding of a poset $P$, the embedding separates the plane into maximal connected components.
Recall that we call the closure of each bounded connected component a region of the embedding.

Given a planar embedding of a lattice $P$, for each $x\in P$, define $S(x)$ to be the union of the chains of $P$ containing $x$ and the horizontal line segments whose endpoints are contained in these chains.
In Figure \ref{fig:S(x)}, the gray region is $S(x)$.
We say that $x$ is \emph{left of} $y$ in the embedding if $y$ is not contained in $S(x)$ and a left-pointing horizontal ray with vertex at~$y$ passes through~$S(x)$.
We similarly define \emph{right of} and note that since $P$ is a lattice, $x$ is left of $y$ if and only if $y$ is right of~$x$.
Furthermore, if $x$ and $y$ are incomparable in $P$, then either $x$ is left of $y$ or $x$ is right of $y$.

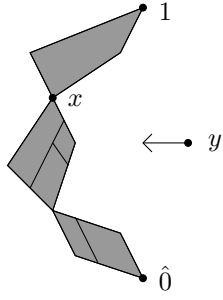
\begin{figure}
\begin{tikzpicture}[scale=.6]
\draw[fill=mgray](6.5, 1)--(5, 1.5)--(4.5, 2.5)--(6,2)--(6.5,1);
\draw[fill=mgray](4.5, 2.5)--(5, 4)--(4.5, 5)--(3.5, 3.5)--(4.5,2.5);
\draw[fill=mgray](4.5, 5)--(4, 6)--(6.5,7)--(6,6)--(4.5, 5);

\draw[fill] (4.5,5) circle (.5ex);
\draw[fill] (7.5,4) circle (.5ex);
\draw[fill] (6.5,1) circle (.5ex);
\draw[fill] (6.5,7) circle (.5ex);
\node at (5,4.95) {$x$};
\node at (8.1,4) {$y$};
\draw (7.5,4)--(6.5, 4);
\draw(6.7, 3.8)--(6.5, 4)--(6.7, 4.2);
\node at (7,1) {$\hat{0}$};
\node at (7,7) {$\hat{1}$};
\draw(6.5, 1)--(5, 1.5)--(4.5, 2.5)--(3.5, 3.5)--(4.5, 5);
\draw(6.5, 1)--(6,2)--(4.5, 2.5);
\draw(5, 2.35)--(5.5, 1.35);
\draw (4.5, 2.5)--(5, 4)--(4.5, 5);
\draw (4.75, 4.5)--(4, 3);
\draw (4.5, 5)--(4, 6)--(6.5,7);
\draw (4.5, 5)--(6, 6)--(6.5,7);
\draw(4.5, 4)--(4.85, 3.5);
\end{tikzpicture}\caption{The shaded region shows $S(x)$.  Since $y$ is not contained in $S(x)$ and the left-pointing horizontal ray with base point at $y$ intersects $S(x)$, we say that $x$ is left of~$y$.}\label{fig:S(x)}
\end{figure}

Let $\mathscr{L}=\{L_1,\ldots,L_l\}$ denote a collection of linear extensions of a poset $P$. 
We say that $\mathscr{L}$ is a \emph{realizer} of $P$ if the intersection of these total orders is $P$.
The \emph{dimension} of a poset $P$ is the size of the smallest realizer.

The following is a well-known result, which we will use to find a realizer of an adjacency poset. 
In the proposition and its proof, given $\sigma =\sigma_1\cdots \sigma_n \in S_n$, we declare $\sigma_i<_{\sigma} \sigma_j$ if and only if $i<j$.
We will routinely pass between a permutation and its associated total order.

\begin{proposition}\label{prop:scapt}
Let $[\sigma,\psi]$ be an interval in the right weak order on $S_n$.
The elements of $[\sigma,\psi]$ are the linear extensions of the intersection of these total orders. 
\end{proposition}
 
 \begin{proof}
Let $\sigma=\sigma_1 \cdots \sigma_n$ and $\psi=\psi_1 \cdots \psi_n$.
Denote the intersection of the total orders $\sigma$ and $\psi$ by $\sigma \cap \psi$.

 Let $u$ be a linear extension of $\sigma\cap \psi$.
 If $(\sigma_i,\sigma_j)\in \text{inv}(\sigma)$ then, since $\sigma\leq\psi$ in the right weak order, $(\sigma_i,\sigma_j)\in \text{inv}(\psi)$.
 Thus $(\sigma_i,\sigma_j)\in \text{inv}(u)$.
 If $(u_i, u_j)\in \text{inv}(u)$, then either $u_j\nless_\sigma u_i$ or $u_j \nless_\psi u_i$.
 Since $\sigma$ and $\psi$ are total orders, we have that $(u_i,u_j)\in \text{inv} (\sigma)$ or $(u_i,u_j)\in \text{inv} (\psi)$.
 In the right weak order $\sigma\leq\psi$, so $(u_i,u_j)\in \text{inv} (\psi)$.
 We conclude that $u\in [\sigma, \psi]$.

 Let $u=u_1 \cdots u_n \in [\sigma,\psi]$ and assume that $u$ is not a linear extension of $\sigma \cap \psi$.
 Thus there exist $i,j \in [n]$ with $i<j$ such 
 that $u_j<_\sigma u_i$ and $u_j<_\psi u_i$.
 If $u_j> u_i$ in numerical order, then $(u_j, u_i)\in \text{inv}(\sigma)$ and $(u_j, u_i)\notin \text{inv}(u)$, contradicting 
 the assumption that $\sigma \leq u$ in the right weak order.
 If $u_j<u_i$ in numerical order, then $(u_i, u_j) \in \text{inv}(u)$ and $(u_i, u_j)\notin \text{inv}(\psi)$ ,
  contradicting the assumption that $u \leq \psi$ in the right weak order.
 Therefore, if $u \in [\sigma,\psi]$, then $u$ is a linear extension of~$\sigma \cap \psi$.
 \end{proof}
 
Since each congruence class of a lattice congruence on the right weak order is an interval \cite[Section~2]{Lat} and since each fiber of $\rho$ is  such a congruence class {\mbox{\cite[Proposition 6.3]{drec}},} each fiber of~$\rho$ is an interval of the right weak order.
Let $D$ be a diagonal rectangulation and let~$L_1$ and~$L_2$ be respectively the minimum and maximum elements in the right weak order on~$S_n$ such that $\rho(L_1)=\rho(L_2)=D$.
By Proposition \ref{prop:scapt}, and since any poset is determined by its set of linear extensions, $\mathscr{L}=\{L_1, L_2\}$ is a realizer of the adjacency poset of $D$.

Given a linear extension $L=\sigma_1\cdots \sigma_n$ of a poset $P$ on $[n]$, let $\pi_L: [n] \to [n]$ be defined by $\pi_L(x)=i$ if and only if $x=\sigma_i$.
The inverse of the permutation $\sigma_1 \cdots \sigma_n$ is $\pi_L(1) \cdots \pi_L(n)$.
If $P$ has realizer $\mathscr{L}=\{L_1,L_2\}$, then the \emph{projection of~$\mathscr{L}$} denoted by $\pi_\mathscr{L}(P)$ is a map from~$[n]$ to $\mathbb{R}^2$ given by $\pi_\mathscr{L}(x)=(\pi_{L_1}(x),\pi_{L_2}(x))$.
This is an embedding of $P$ into the componentwise order on $\mathbb{R}^2$. 
To view this embedding of~$P$ as a Hasse diagram for $P$, we take ``up" to be the direction of the vector $\langle1,1\rangle$.

\begin{theorem}[{\cite[p. 69]{Trotter}}]\label{thm:planar}
If $P$ is a lattice with realizer $\mathscr{L}=\{L_1, L_2\}$, then the embedding of~$P$ into the componentwise order on $\mathbb{R}^2$ given by $\pi_{\mathscr{L}}(P)$ is a planar embedding of~$P$.
\end{theorem}

The following proposition is \cite[p 32, Exercise 7(a)]{Birkhoff}.  
Since every Baxter poset is finite, bounded, and has a planar embedding, this proposition implies that every Baxter poset is a lattice.

\begin{proposition}\label{prop:lattice}
A finite planar poset $P$ is a lattice if and only if $P$ is bounded. 
\end{proposition}

The following lemma is \cite[Lemma 2.1]{Bjorner}: 
\begin{lemma}\label{lemma:join}
Let $P$ be a bounded poset such that every chain of $P$ is of finite length.
If, for any $x$ and $y$ in $P$ such that $x$ and $y$ both cover some element $z$, the join $x \vee y$ exists, then~$P$ is a lattice.
\end{lemma}

We now have the necessary tools to prove our main result.

\begin{proof}[(Proof of Theorem \ref{thm:adjposet})]
Let $D$ be a diagonal rectangulation of size $n$ and $P$ the associated adjacency poset.
We first demonstrate that $P$ satisfies the five conditions of Definition \ref{def:baxposet}.
The rectangle $x$ of $D$ whose lower-left corner coincides with the lower-left corner of $S$ contains interior points below and left of interior points of all other rectangles of $D$.
Thus for every $y \in [n]-\{x\}$, we have that $x<_P y$ . 
Similarly, the label of the rectangle of $D$ whose upper-right corner coincides with the upper-right corner of $S$ is greater, in $P$, than every other element of $P$.
Therefore, $P$ is a bounded poset.

Observe that any rectangle $x$ of $D$ is the left rectangle of at most one of the configurations shown in Figure \ref{fig:adjcovers} and the bottom rectangle of at most one of the configurations shown in Figure \ref{fig:adjcovers}.
Thus, $x$ is covered by at most two elements of~$P$.
Similarly, $x$ covers at most two elements of $P$. 

To show that $P$ meets Condition 3 of Definition \ref{def:baxposet}, for a contradiction assume that~$P$ contains a $2$-$14$-$3$, a $3$-$14$-$2$, a $2$-$41$-$3$ or a $3$-$41$-$2$ chain.
This implies that some linear extension~$\sigma$ of $P$ contains this pattern with the $``4"$ and $``1"$ adjacent.
By Proposition \ref{prop:equivclass}, transposing the $``4"$ and $``1"$ in this linear extension results in a permutation $\sigma'$ such that $\rho(\sigma)=\rho(\sigma')$.
Since the fiber $\rho^{-1}(D)$ is the set of linear extensions of $P$, the permutation $\sigma'$ is also a linear extension of $P$.
However, this contradicts the assumption that the ``4" and the ``1" are related in $P$.   

Since the labeling of the rectangles of $D$ comes from the map $\rho$ from permutations to diagonal rectangulations, to demonstrate that $P$ meets Condition 4 of Definition~\ref{def:baxposet}, we rely on observations about this map.  
Consider an interval $[x,y]$ of $P$ such that $(x,y)$ is disconnected.
There exist~$x_r\neq x_a$ such that $x\lessdot_P x_r$ and $x\lessdot_P x_a$.  
By Theorem \ref{thm:adjcovers}, since no four rectangles of a diagonal rectangulation share a vertex, rectangles $x, x_a,$ and $x_r$ form one of the configurations shown in Figure \ref{fig:intervalpf}. 
In Diagram (i), the left side of rectangle $x_a$ is missing to indicate that the lower-left vertex of rectangle $x_a$ coincides with or is left of the upper-left vertex of rectangle $x$.  
The bottom of rectangle $x_r$ is missing in Diagram (ii)  to similarly indicate that the lower-left vertex of rectangle $x_r$ coincides with or is below the lower-right vertex of rectangle $x$.

\begin{figure}
\begin{center}
\begin{tikzpicture}[scale=.6]
\draw (10,0) rectangle (12,2);
\draw(12,0) rectangle (13,1);
\draw (10,2) rectangle (11,3);
\node at (11,1) {$x$};
\node at (10.5,2.5) {$x_a$};
\node at (12.5, .5) {$x_r$};
\node at (11.5,-1) {$(iii)$};

\draw (5,0) rectangle (7,2);
\draw (7,2)--(8,2)--(8,0);
\draw (5,2) rectangle (6,3);
\node at (6,1) {$x$};
\node at (5.5,2.5) {$x_a$};
\node at (7.5, 1) {$x_r$};
\node at (6.5,-1) {$(ii)$};

\draw(0,0) rectangle (2,2);
\draw (2,2)--(2,3)--(0,3);
\draw (2,0) rectangle (3,1);
\node at (1,1) {$x$};
\node at (1, 2.5) {$x_a$};
\node at (2.5, .5) {$x_r$};
\node at (1.5,-1) {$(i)$};

\end{tikzpicture}
\end{center}
\caption{Given that $x \lessdot_P x_a$ and $x\lessdot_P x_r$ with $x_a\neq x_r$, in diagonal rectangulation $D$ rectangles $x, x_a$ and $x_r$ form one of the three configurations shown.}
\label{fig:intervalpf}
\end{figure}
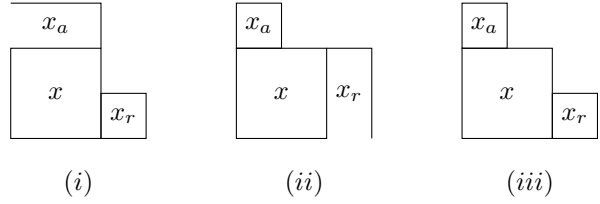

First assume that rectangles $x, x_a,$ and $x_r$ are in the configuration shown in Diagram (i) of Figure \ref{fig:intervalpf} and let $W$ be the vertical wall on the right side of rectangle~$x$.
The lower-right vertex of rectangle $x$ and the lower-left vertex of rectangle~$x_r$ coincide, so rectangle~$x$ is the lowermost rectangle on the left side of $W$.
 By the definition of $\rho$, rectangle $x+1$ is the uppermost rectangle adjacent to the right side of $W$ and the lower-left corner of rectangle $x+1$ is below the upper-right corner of rectangle $x$.
Since the interiors of the right edge of rectangle~$x_a$ and the left edge of rectangle $x+1$ intersect, we have that $x_a<_P x+1$.
Since the upper-right corner of rectangle $x+1$ is strictly right of $W$ and above rectangle~$x_r$, we have that $x_r<_P x+1$.
We wish to show that $x+1=y$, i.e.,~there does not exist $z<_P x+1$ such that $x_a<_P z$ and $x_r<_P z$.
We will prove a stronger statement: $x_a \vee x_r$ exists and $x_a \vee x_r=x+1$.
Since $x+1$ is an upper bound for $x_a$ and~$x_r$, it suffices to demonstrate that any other upper bound $z$ satisfies $x+1\leq_P z$.
To obtain a contradiction, assume that $x+1 \nleq_P z$ for some upper bound $z$. 
We use an argument similar to the argument used in the proof of Theorem~\ref{thm:adjcovers}.
Since $x<_Pz$, we have that $z\nleq_Px$.
Thus rectangle~$z$ contains no interior points that are both left of the vertical line containing $W$ and below the horizontal line containing the top of rectangle~$x$.
Since $x+1 \nleq_P z$, rectangle $z$ contains no interior points that are both right of the vertical line containing $W$ and above the horizontal line containing the bottom of rectangle~$x+1$.
Thus $z$ is contained in either the region left of the vertical line containing~$W$ and above the horizontal line containing the top of rectangle $x$ or the region right of the vertical line containing $W$ and below the horizontal line containing the bottom of rectangle $x+1$.
Note that these regions are disjoint, that rectangle $x_a$ is contained in the first region, and that rectangle $x_r$ is contained in the second region.
 In $P$, the label of a rectangle contained in the first region cannot cover the label of a rectangle contained in the second region and vice versa.
Thus $x_a\nless_P z$ or $x_r\nless_P z$, a contradiction.
Therefore $x_a\vee x_r=x+1$.

 When rectangles $x, x_a$ and $x_r$ form the configuration shown in Diagram (ii) of Figure~\ref{fig:intervalpf}, by considering the horizontal wall $W$ above rectangle~$x$ and the rightmost rectangle above~$W$, rectangle~$x-1$, we similarly show that $y=x-1$ and that $x_a\vee x_r=x-1$. 
 In the case illustrated in Diagram~(iii) of Figure~\ref{fig:intervalpf}, we first observe that since~$D$ is a diagonal rectangulation, the wall above or on the right side of rectangle~$x$ extends beyond the upper-right corner of rectangle $x$.
 In either case, using the previous arguments, we show that $y=x+1$ or $y=x-1$ and~$y=x_a\vee x_r$.

To demonstrate that $P$ meets Condition 5 of Definition \ref{def:baxposet}, note that by Condition~1 of the definition, and since we verified that $y=x_a \vee x_r$ in each case of the proof of Condition 4, Lemma \ref{lemma:join} implies that $P$ is a lattice.
Let $L_1$ and~$L_2$ be respectively the minimum and maximum elements in the right weak order on~$S_n$ such that $\rho(L_1)=\rho(L_2)=D$.
By Proposition \ref{prop:scapt}, $\mathscr{L}=\{L_1, L_2\}$ is a realizer of~$P$. 
By Theorem \ref{thm:planar}, the Hasse diagram obtained from $\pi_{\mathscr{L}}(P)$ is a planar embedding of~$P$.
Let $[x,y]$ be an interval of $P$ such that~$(x,y)$ is disconnected.
Let $x\lessdot_P x_l$ and $x \lessdot_P x_r$ where $x_l$ is left of $x_r$ in the planar Hasse diagram obtained from~$\pi_{\mathscr{L}}(P)$.
 Let $\pi_\mathscr{L}(x_l)=(a,b)$ and  $\pi_\mathscr{L}(x_r)=(c,d)$.
Since~$x_l$ and $x_r$ are incomparable with $x_l$ left of $x_r$ in the planar Hasse diagram, we have that $a<c$ and $b>d$ in numerical order.
This implies that $x_l$ precedes $x_r$ in $L_1$ and $x_l$ follows $x_r$ in~$L_2$.
Since $L_1\leq L_2$ in the right weak order, $(x_l, x_r)\in \text{inv}(L_2)$.
Thus  $x_l<x_r$ in numerical order.

Rectangles $x, x_l$, and $x_r$ form one of the configurations shown in Figure \ref{fig:intervalpf} (with~$x_l$ replacing $x_a$).
In every diagram of Figure \ref{fig:intervalpf}, since each rectangle~$x_i$ such that ${x_l\leq_P x_i <_P y}$ is contained in the region above the horizontal line containing the top of rectangle~$x$ and left of the vertical line containing the left side of rectangle~$y$, rectangle~$x_i$ intersects the diagonal of $S$ in that region.
This implies that $x_i<x$ in numerical order.
Additionally, for each~$x_j$ such that $x_r\leq_P x_j <_P y$, since rectangle~$x_j$ intersects the diagonal of $D$ in the region right of the vertical line containing the right side of rectangle $x$ and below the horizontal line containing the bottom of rectangle~$y$, we have that $x<x_j$ in numerical order.
Thus one connected component of $(x,y)$ contains elements numerically smaller than $x$ and $y$ while the other connected component contains elements numerically larger than~$x$ and $y$.
Since $x_l<x_r$ in numerical order with~$x_l$ contained in the left component of $(x,y)$ and $x_r$ contained in the right component, given $w,z\in (x,y)$ such that $w$ is left of $z$ in this planar embedding of $P$, we have that $w<x<z$ in numerical order.

We have shown that the adjacency poset $P$ satisfies each of the conditions in Definition \ref{def:baxposet}, so $P$ is a Baxter poset.

Now let $P$ be a Baxter poset.
To demonstrate that~$P$ is an adjacency poset, we first show that the set of linear extensions of $P$ is a union of fibers of $\rho$.
In what follows, we assume that $P$ is embedded as described in Condition 5 of Definition~\ref{def:baxposet}.  
Let $\sigma=\sigma_1 \cdots \sigma_n$ be a linear extension of~$P$ and suppose $\psi=\sigma_1\cdots \sigma_{j-1} \sigma_{j+1}\sigma_j \sigma_{j+2}\cdots \sigma_n$ such that $\rho(\sigma)=\rho(\psi)$.
We will show that $\psi$ is also a linear extension of~$P$.
Since $\rho(\sigma)=\rho(\psi)$ and~$\sigma\lessdot \psi$ or $\psi \lessdot \sigma$ in the right weak order, by Proposition \ref{prop:equivclass}, the permutations $\sigma$ and $\psi$ are related by a single ($2$-$41$-$3$ $\leftrightarrow$ $2$-$14$-$3$) or ($3$-$41$-$2$ $\leftrightarrow$ $3$-$14$-$2$) move.
Let $a \sigma_j \sigma_{j+1} b$ be an occurrence of one of these four patterns in $\sigma$ such that swapping $\sigma_j$ and $\sigma_{j+1}$ is a move.
Since~$\sigma$ is a linear extension of $P$, the permutation $\psi$ is also a linear extension of~$P$ if and only if~$\sigma_j$ and $\sigma_{j+1}$ are incomparable in~$P$.
To proceed via contradiction, assume that $\sigma_j$ and $\sigma_{j+1}$ are comparable in~$P$.
Because~$\sigma_j$ precedes $\sigma_{j+1}$ in $\sigma$ and $\sigma$ is a linear extension of $P$, we have that $\sigma_{j+1}\nless_P \sigma_j$.
Thus $\sigma_j<_P \sigma_{j+1}$.
This implies that $\sigma_j \lessdot_P \sigma_{j+1}$ since any $\sigma_k$ such that $\sigma_j<_P \sigma_k <_P \sigma_{j+1}$ would be between $\sigma_j$ and~$\sigma_{j+1}$ in every linear extension of $P$ (and in particular in $\sigma$).  
By Condition~3 of Definition \ref{def:baxposet}, at least one of $\{a,b\}$ is incomparable with at least one of $\{\sigma_j, \sigma_{j+1}\}$.
We assume that $a$ is incomparable with $\sigma_j$ or $\sigma_{j+1}$ and note that if~$b$ is instead incomparable with~$\sigma_j$ or $\sigma_{j+1}$, then the argument is analogous.
Since~$a$ precedes~$\sigma_j$ in $\sigma$, our assumption implies that either $a<_P \sigma_{j+1}$ and $a$ and $\sigma_j$ are incomparable, or $a$ is incomparable with both~$\sigma_j$ and $\sigma_{j+1}$.
In either case, $a$ and $\sigma_j$ are incomparable.

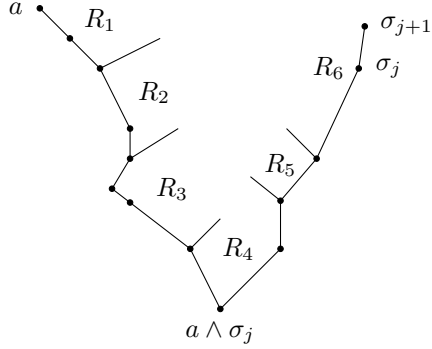
\begin{figure}
\begin{tikzpicture}[scale=.8]
\draw (0,0)--(1,-1)--(1.5, -2)--(1.5, -2.5)--(1.2,-3)--(2.5, -4)--(3, -5);

\node at (-.4,0) {$a$};
\node at (3,-5.4) {$a \wedge \sigma_j$};
\node at (5.8, -1) {$\sigma_j$};
\node at (6.1, -.3) {$\sigma_{j+1}$};
\draw (3,-5)--(4,-4)--(4,-3.2)--(4.6, -2.5)--(5.3, -1)--(5.4, -.3);

\draw[fill] (0,0) circle (.3ex);
\draw[fill] (1,-1) circle (.3ex);
\draw[fill] (1.5,-2) circle (.3ex);
\draw[fill] (1.5,-2.5) circle (.3ex);
\draw[fill] (1.2, -3) circle (.3ex);
\draw[fill] (2.5, -4) circle (.3ex);
\draw[fill] (3,-5) circle (.3ex);
\draw[fill] (4,-4) circle (.3ex);
\draw[fill] (4,-3.2) circle (.3ex);
\draw[fill] (4.6, -2.5) circle (.3ex);
\draw[fill] (5.3, -1) circle (.3ex);
\draw[fill] (.5, -.5) circle (.3ex);
\draw[fill] (1.5, -3.23) circle (.3ex);
\draw[fill] (5.4, -.3) circle (.3ex);

\draw(1,-1)--(2, -.5); 
\draw (1.5, -2.5)--(2.3, -2);
\draw(2.5, -4)--(3,-3.5);
\draw(4,-3.2)--(3.5, -2.8);
\draw (4.6, -2.5)--(4.1, -2);

\node at (1, -.2) {$R_1$};
\node at (1.9, -1.4) {$R_2$};
\node at (2.2, -3) {$R_3$};
\node at (3.3, -4) {$R_4$};
\node at (4, -2.6) {$R_5$};
\node at (4.8, -1) {$R_6$};
\end{tikzpicture}
\caption{An illustration for the proof of Theorem \ref{thm:adjposet}.} 
\label{fig:regions}
\end{figure}

By Proposition~\ref{prop:lattice},~$P$ is a lattice so we may consider $S(a)$ and $S(\sigma_j)$.
First assume that $a$ is left of $\sigma_j$ 
and consider the maximal chain $C_1$ of $P$ from $a$ to the minimal element of $P$, denoted~$\hat{0}$, that follows the right boundary of $S(a)$.
Let~$C_2$ denote the maximal chain of $P$ from~$\sigma_j$ to~$\hat{0}$ that follows the left boundary of $S(\sigma_j)$.  
Note that~$C_1$ and~$C_2$ intersect at $a \wedge \sigma_j$ and let~$C_1'$ and~$C_2'$ denote the chains from~$a$ and~$\sigma_j$ to $a \wedge \sigma_j$ obtained by truncating $C_1$ and~$C_2$ respectively.
Figure~\ref{fig:regions} shows an example of the chains~$C_1'$ and~$C_2'$.
Each edge of~$C_1'$ and $C_2'$ is the edge of a region of~$P$ that lies right of $C_1'$ and left of $C_2'$.  
Starting at~$a$, traveling down $C_1'$ to $a \wedge \sigma_j$, label the sequence of regions right of and adjacent to~$C_1'$ with $R_1,\ldots,R_l$.
Starting at $a\wedge \sigma_j$, and traveling up $C_2'$ to $\sigma_j$, continue by labeling the sequence of regions left of and adjacent to $C_2'$  with $R_l,R_{l+1},\ldots,R_m$.  
In Figure~\ref{fig:regions}, $l=4$ and $m=6$.
For each $i\in [m-1]$, by Condition 2 of Definition \ref{def:baxposet}, the region~$R_i$ shares an edge with the region $R_{i+1}.$ 
(Otherwise $C_1$ is not the right boundary of~$S(a)$ or $C_2$ is not the left boundary of $S(\sigma_j)$.)
Since $P$ is a lattice, for $i \in [m]$, each region~$R_i$ has a minimal element, denoted $r_i$ contained in the boundary of~$R_i$.
Each~$r_i$ is a vertex of $C_1' \cup C_2'$.
(If the edge of some region is contained in $C_1' \cup C_2'$ and that region's minimal element is not on $C_1' \cup C_2'$, then again either~$C_1$ is not the right boundary of $S(a)$ or $C_2$ is not the left boundary of~$S(\sigma_j)$.)
For each $i\in [l]$, the minimal element $r_i$ is contained in the left side of the region $R_{i+1}$.
Thus, by Condition 5 of Definition \ref{def:baxposet}, we have that $a< r_1< \cdots < r_l=a \wedge \sigma_j$ in numerical order.
For each $i\in \{l+1,\ldots,m\}$, the minimal element $r_i$ is contained in the right side of region~$R_{i-1}$.
Thus $a \wedge \sigma_j=r_l<r_{l+1} <\cdots <r_{m}<\sigma_j$ in numerical order.
Combining these strings of inequalities, we conclude that $a<\sigma_j$ in numerical order.

In a similar way, construct a sequence of regions $S_1,\ldots,S_p$ using the section of the right boundary of $S(a)$ from $a$ to $a \vee \sigma_{j+1}$ and the section of the left boundary of $S(\sigma_{j+1})$ from~$\sigma_{j+1}$ to $a\vee \sigma_{j+1}$.  
If $a<_P\sigma_{j+1}$, then $a \vee \sigma_{j+1}=\sigma_{j+1}$.
Whether $a<_P \sigma_{j+1}$ or $a$ and $\sigma_{j+1}$ are incomparable in $P$, using the sequence of maximal elements of these regions together with Condition 5 of Definition \ref{def:baxposet}, we obtain a chain of inequalities and conclude that $a<\sigma_{j+1}$ in numerical order.
However, combining the conclusions that  $a<\sigma_j$ and $a<\sigma_{j+1}$ contradicts to the assumption that $a\sigma_j\sigma_{j+1}b$ is an occurrence of a $2$-$41$-$3$, a $2$-$14$-$3$, a $3$-$41$-$2$ or a $3$-$14$-$2$ pattern.

If $\sigma_j$ is left of $a$ in $P$, then to construct sequence of regions $R_1,\ldots,R_m$, let~$C_1$ be the right boundary of $S(\sigma_j)$ and $C_2$ be the left boundary of $S(a)$.
To construct the sequence of regions $S_1,\ldots,S_p$, use the right boundary of $S(\sigma_{j+1})$ and the left boundary of $S(a)$.
Using these sequences and the corresponding chains of inequalities, we conclude that in numerical order $\sigma_j<a$ and $\sigma_{j+1}<a$.
This conclusion again contradicts the assumption that $a\sigma_j\sigma_{j+1}b$ is an occurrence of a $2$-$41$-$3$, a $2$-$14$-$3$, a $3$-$41$-$2$ or a $3$-$14$-$2$ pattern.
In both cases, we see that $\sigma_j$ and~$\sigma_{j+1}$ are incomparable in $P$.
Therefore the set of linear extensions of $P$ is a union of fibers of $\rho$.

Any two linear extensions of a poset are related by a sequence of adjacent transpositions.
Consider two linear extensions $\sigma$ and $\psi$ of $P$ that differ by an adjacent transposition.
To complete the proof that $P$ is an adjacency poset, we will show that $\rho(\sigma)=\rho(\psi)$.
Specifically, we demonstrate that $\sigma$ and $\psi$ are related by a ${(2\text{-}41\text{-}3} \leftrightarrow$ $2$-$14$-$3$) or ($3$-$41$-$2$ $\leftrightarrow$ $3$-$14$-$2$) move.
Suppose that $\sigma=\sigma_1\cdots \sigma_j \sigma_{j+1} \cdots \sigma_n$ and $\psi=\sigma_1\cdots \sigma_{j-1}\sigma_{j+1} \sigma_{j}\sigma_{j+2} \cdots \sigma_n$.
Since $\sigma_j$ precedes $\sigma_{j+1}$ in $\sigma$ but $\sigma_{j+1}$ precedes $\sigma_j$ in $\psi$, we have that $\sigma_j$ and $\sigma_{j+1}$ are incomparable in $P$.
This implies that $\sigma_j\wedge \sigma_{j+1}\notin \{\sigma_j,\sigma_{j+1}\}$ and $\sigma_j\vee \sigma_{j+1}\notin \{\sigma_j,\sigma_{j+1}\}$.
Without loss of generality, up to swapping $\sigma$ and $\psi$, we can assume that $\sigma_j$ is left of $\sigma_{j+1}$ in $P$.
Consider sequences of regions $R_1,\ldots,R_m$ and $S_1,\ldots,S_p$, defined as in the previous paragraph, replacing~$a$ with $\sigma_{j}$.
Using these sequences of adjacent regions and the resulting inequalities, we obtain $\sigma_j<\sigma_j\wedge \sigma_{j+1}<\sigma_{j+1}$ and $\sigma_j<\sigma_j \vee \sigma_{j+1} < \sigma_{j+1}$ in numerical order.
By definition, $\sigma_j\wedge \sigma_{j+1}<_P \sigma_j$ and $\sigma_j\wedge \sigma_{j+1}<_P \sigma_{j+1}$, so $\sigma_j\wedge \sigma_{j+1}$ precedes~$\sigma_j$ and~$\sigma_{j+1}$ in $\sigma$ and $\psi$.
Similarly,~$\sigma_j$ and $\sigma_{j+1}$ precede $\sigma_j\vee \sigma_{j+1}$ in $\sigma$ and $\psi$. 
Thus the sequence $(\sigma_j\wedge \sigma_{j+1}) \sigma_j \sigma_{j+1} (\sigma_j\vee \sigma_{j+1})$ is a $2$-$41$-$3,$ a $2$-$14$-$3,$ a $3$-$41$-$2,$ or a $3$-$14$-$2$ pattern in $\sigma$.
\end{proof}

\section{Twisted Baxter and Baxter Permutations from Baxter Posets}\label{perms}
Let $P$ be a poset.
We say that a subset $I$ of the elements of $P$ is an \emph{order ideal of $P$} if and only if for every $a\in I$, if $b<_P a$, then $b \in I$.
We say that an ordering $a_1 \cdots a_i$ of a subset of the elements of $P$ is a \emph{partial linear extension} of~$P$ if $\{a_1,\ldots,a_j\}$ is an order ideal of $P$ for all $j\in [i]$. 
Given a poset $P$ on $[n]$, the permutation~$\sigma$ is a linear extension of $P$ if and only if $\sigma$ satisfies the definition of a partial linear extension.  
Given a partial linear extension $\sigma_1\cdots \sigma_{i-1}$ of $P$, we define $A_i\subseteq[n]$ by $u\in A_i$ if and only if $\sigma_1 \cdots \sigma_{i-1}u$ is a partial linear extension of $P$.
We label this set $A_i$ because it forms an antichain (a set of pairwise incomparable elements) of $P$.

\begin{theorem}
Given a Baxter poset $P$, the unique twisted Baxter permutation $\sigma=\sigma_1\cdots \sigma_n$ that is a linear extension of $P$ is constructed by choosing $\sigma_i=\min(A_i)$ for each $i\in [n]$.
\end{theorem}

Note that $\min(A_i)$ denotes the smallest, in numerical order, element of $A_i$.
If a Baxter poset $P$ is given a natural embedding, then this selection is equivalent to choosing the leftmost (in the embedding) element of $A_i$ for each $i \in [n]$.

\begin{proof}
Let $P$ be a Baxter poset and $D$ the associated diagonal rectangulation.
By Theorem~\ref{thm:adjposet}, the total order $\sigma$ is a linear extension of $P$ if and only if $\rho(\sigma)=D$.
Since $\rho$ restricts to a bijection between diagonal rectangulations and twisted Baxter permutations (Theorem \ref{thm:rho}), there is a unique linear extension $\sigma=\sigma_1\cdots \sigma_n$ of $P$ that is a twisted Baxter permutation.
To construct~$\sigma$ one entry at a time, we must describe a method for choosing $\sigma_i$ from $A_i$.
By Proposition~\ref{prop:minmax}, the permutation~$\sigma$ is the minimal element of the right weak order such that $\rho(\sigma)=D$.
That is, $\sigma$ is the linear extension of $P$ that contains the fewest inversions. 
Therefore, $\sigma_i=\min(A_i)$ for all~$i\in [n]$.
\end{proof}

The following results will be used in the proof of Theorem \ref{thm:baxter}.
The next lemma is equivalent to Corollary 4.2 in \cite{drec} which states that $\sigma$ is a Baxter permutation if and only if $\sigma^{-1}$ is a Baxter permutation.

\begin{lemma}\label{lemma:4-1to41} 
The permutation $\sigma$ is a Baxter permutation if and only if $\sigma$ contains no subsequence $\sigma_i\sigma_j \sigma_k \sigma_l$ such that $|\sigma_l-\sigma_i|=1$ and the subsequence is an occurrence of the pattern 2-4-1-3 or the pattern 3-1-4-2.
\end{lemma}

By Theorem \ref{thm:rho}, there exists a unique linear extension of $P$ that is a Baxter permutation.

\begin{lemma}\label{lemma:xnonempty}
Let $P$ be a Baxter poset and $\sigma$ be the unique Baxter permutation that is a linear extension of $P$.
Then $\sigma$ respects the arrows of $P$.
\end{lemma}

\begin{proof}
Let $P$ be a Baxter poset with a natural embedding.  
Let $\sigma$ denote a linear extension that does not respect the arrow of some region $R$ of $P$.
Let $\min_R$ and $\max_R$ respectively denote the minimal and maximal elements of $R$.  
By Condition 4 of Definition \ref{def:baxposet}, we have that $\min_R$ and $\max_R$ differ in value by one.
Since $\sigma$ does not respect the arrow of $R$, there exists a subsequence $\min_R \sigma_i \sigma_j \max_R$ of $\sigma$ such that $\sigma_i$ and $\sigma_j$ 
are contained in the boundary of $R$, one of these contained in the left component of $(\min_R, \max_R)$ and the other contained in the right component of $(\min_R, \max_R)$, and this subsequence is an occurrence of a 2-4-1-3 or a 3-1-4-2 pattern.
Thus, by Lemma \ref{lemma:4-1to41}, $\sigma$ is not a Baxter permutation.
\end{proof}

We make several useful observations about the map $\rho$.
Given a diagonal rectangulation~$D$, if $W$ is a horizontal wall of $D$ and rectangle $a$ is the leftmost rectangle below and adjacent to $W$, then rectangle $a-1$ is the rightmost rectangle above and adjacent to $W$ and $a$ precedes $a-1$ in every permutation $\sigma$ such that $\rho(\sigma)=D$.
Each rectangle below and adjacent to $W$ has a label larger than $a$ and each rectangle above and adjacent to $W$ has a label smaller than $a-1$.
Similarly, if $W$ is a vertical wall of $D$ and rectangle $a$ is the lowermost rectangle left of and adjacent to $W$, then rectangle $a+1$ is the uppermost rectangle right of and adjacent to $W$ and~$a$ precedes $a+1$ in every permutation $\sigma$ such that $\rho(\sigma)=D$.
Additionally, every rectangle left of and adjacent to $W$ has label smaller than $a$ and every rectangle right of and adjacent to $W$ has label larger than $a+1$.

The lemma below follows from the definition of a Baxter permutation, the above observations, and Lemma \ref{lemma:4-1to41}.

\begin{lemma}\label{lemma:wallcond}
Let $D$ be a diagonal rectangulation and $\sigma=\sigma_1\cdots \sigma_n\in S_n$ such that $\rho(\sigma)=D$.
If $\sigma$ is a Baxter permutation, then $\sigma$ satisfies the following properties:
\begin{itemize}
\item If rectangles $\sigma_i$ and $\sigma_j$ are adjacent to a horizontal wall $W$ with rectangle~$\sigma_i$ below $W$ and rectangle $\sigma_j$ above $W$, then $\sigma_i$ precedes $\sigma_j$ in $\sigma$ and 
\item If rectangles $\sigma_i$ and $\sigma_j$ are adjacent to a vertical wall $W$ with rectangle $\sigma_i$ left of $W$ and rectangle $\sigma_j$ right of $W$, then $\sigma_i$ precedes $\sigma_j$ in $\sigma$.
\end{itemize}
\end{lemma}

To complete the proof of Theorem \ref{thm:baxter},  we refer to a second family of rectangulations, called generic rectangulations.
We need generic rectangulations exclusively to prove Lemma \ref{lemma:prop}, a lemma about diagonal rectangulations, so we only provide the required background related to generic rectangulations from \cite{grec}.
We say that a rectangulation $R$ is a \emph{generic rectangulation} if and only if there exists no set of four rectangles of $R$ that share a vertex.
The set of diagonal rectangulations with~$n$ rectangles is a subset of the set of generic rectangulations with $n$ rectangles.

As with diagonal rectangulations, there is a map $\gamma$ that takes a permutation on~$[n]$ to a generic rectangulation of size $n$ (see \cite[Section 3]{grec}) and restricts to a bijection between a subset of $S_n$ and generic rectangulations containing $n$ rectangles.
We will not need a complete description of $\gamma$, so we instead quote the required results.

We say that a permutation $\sigma=\sigma_1\cdots \sigma_n$ \emph{avoids the pattern 2-4-51-3 and the pattern $4\text{-}2\text{-}51\text{-}3$} if there does not exist $i<j<k<k+1<l$ such that $\sigma_{k+1}<\sigma_i<\sigma_l<\sigma_j<\sigma_k$ or $\sigma_{k+1}<\sigma_j<\sigma_l<\sigma_i<\sigma_k$.
Similary, we say that $\sigma$ \emph{avoids the pattern 3-51-2-4 and the pattern 3-51-4-2} if there does not exist $i<j<j+1<k<l$ such that $\sigma_{j+1}<\sigma_k<\sigma_i<\sigma_l<\sigma_j$ or $\sigma_{j+1}<\sigma_l<\sigma_i<\sigma_k<\sigma_j$.

\begin{theorem}[{\cite[Theorem 4.1]{grec}}]
The map $\gamma$ restricts to a bijection between permutations of $[n]$ that avoid the patterns $\{$2-4-51-3, 4-2-51-3, 3-51-2-4, 3-51-4-2$\}$ and generic rectangulations containing $n$ rectangles.
\end{theorem}

We say that two permutations $\sigma$ and $\psi$ are related by a ($2$-4-$15$-$3 \leftrightarrow 2$-4-$51$-$3$) \emph{move} if one of these permutations contains a subsequence $\sigma_{i_1} \sigma_{i_2} \sigma_{i_3} \sigma_{i_4}\sigma_{i_5}$ that is an occurrence of the pattern 2-4-51-3 and switching the positions of the adjacent entries $\sigma_{i_3}$ and $\sigma_{i_4}$ in that permutation results in the other permutation.  
We say that~$\sigma$ and $\psi$ are related by a \emph{${(4\text{-}2\text{-}15\text{-}3} \leftrightarrow {4\text{-}2\text{-}51\text{-}3)}$ move}, a \emph{$(3$-$15$-$2$-$4 \leftrightarrow 3$-$51$-$2$-$4)$ move}, or a \emph{$(3$-$15$-$4$-$2\leftrightarrow 3$-$51$-$4$-$2)$ move} if~$\sigma$ and $\psi$ satisfy the analogous conditions with these patterns. 

\begin{proposition}[{\cite[Proposition 4.3]{grec}}]\label{prop:equivclassgrec}
Two permutations $\sigma$ and $\psi$ satisfy $\gamma(\sigma)=\gamma(\psi)$ if and only if they are related by a sequence of $(2$-$4$-$15$-$3 \leftrightarrow 2$-$4$-$51$-$3)$ moves, ${(4\text{-}2\text{-}15\text{-}3} \leftrightarrow 4$-$2$-$51$-$3)$ moves, $(3$-$15$-$2$-$4 \leftrightarrow 3$-$51$-$2$-$4)$ moves, and $(3$-$15$-$4$-$2\leftrightarrow 3\text{-}51\text{-}4\text{-}2)$ moves. 
\end{proposition}

The map~$\gamma$ labels each rectangle of the constructed generic rectangulation with an element of~$[n]$.
Given a generic rectangulation $R$, this labeling of rectangles is unique i.e., if $x, y \in S_n$ such that $\gamma(x)=\gamma(y)$, then the labeling of the rectangles obtained from $\gamma(x)$ agrees with the labeling of the rectangles obtained from $\gamma(y)$.
Thus we can refer to the rectangle of $R$ with label $i$ as \emph{rectangle $i$}.

Given a generic rectangulation $R$ and a wall $W$ of $R$, the \emph{wall shuffle of $W$}, denoted $\sigma_W$, records the order in which the rectangles adjacent to $W$ appear along~$W$.
Specifically, let~$W$ be a horizontal wall of $R$.
To find the wall shuffle of $W$, temporarily label each vertex contained in $W$ as follows.
If the vertex is the upper-left vertex of some rectangle $x$, then label the vertex with $x$.
Otherwise, the vertex is the lower-right vertex of some rectangle~$y$, and we label it with $y$.
The left-to-right ordering of the vertices along $W$ provides an ordering of these vertex labels, and this ordering is $\sigma_W$.
Similarly, if $W$ is a vertical wall of $R$, we temporarily label the vertices contained in $W$.
We label a vertex with $x$ if it is the lower-right vertex of rectangle~$x$.
Otherwise, the vertex is the upper-left vertex of some rectangle~$y$ and we label the vertex with $y$.
The bottom-to-top order of these labels along $W$ gives us~$\sigma_W$.

The map $\gamma$ constructs a generic rectangulation $R$ from a permutation in two steps.
Given~$\sigma \in S_n$, we first construct $\rho(\sigma)$.
Then, for each wall of $\rho(\sigma)$, the vertices are labeled as described above.
Finally, the vertices (and the attached edges) are reordered along each wall so that the wall shuffle of each wall is a subsequence of $\sigma$.
For us, the key point is that, to specify a generic rectangulation, it suffices to identify the associated diagonal rectangulation and an order of the vertices along each wall (i.e. a wall shuffle for each wall).

Given a Baxter permutation $\sigma$, the conditions given in Lemma \ref{lemma:wallcond} specify the wall shuffles of the associated generic rectangulation $\gamma(\sigma)$.
As a result, we can make use of generic rectangulations to prove the following lemma.

\begin{lemma}\label{lemma:prop}
Let $D$ be a diagonal rectangulation.  
Then there is a unique permutation $\sigma$ such that $\rho(\sigma)=D$ and such that $\sigma$ satisfies the properties given in Lemma~\ref{lemma:wallcond}.
This permutation $\sigma$ is the Baxter permutation associated with $D$.
\end{lemma}

\begin{proof}
Let $D$ be a diagonal rectangulation and $\sigma$ the unique Baxter permutation such that $\rho(\sigma)=D$.
The permutation $\sigma$ satisfies the properties given in Lemma~\ref{lemma:wallcond}.
Assume that there exists a second permutation $\psi$ such that $\rho(\psi)=D$ and $\psi$ satisfies the properties given in Lemma~\ref{lemma:wallcond}.
Since $\rho(\sigma)=\rho(\psi)$ and the wall shuffles of $\gamma(\sigma)$ agree with the wall shuffles of~$\gamma(\psi)$, we have that $\gamma(\sigma)=\gamma(\psi)$.
Thus, by Proposition~\ref{prop:equivclassgrec}, the permutations $\sigma$ and $\psi$ are related by a sequence of adjacent transpositions in which each transposition is a (2-4-15-3$\leftrightarrow$ 2-4-51-3) move, a ${(4\text{-}2\text{-}15\text{-}3} \leftrightarrow$ 4-2-51-3) move, a (3-15-2-4 $\leftrightarrow$ 3-51-2-4) move, or a (3-15-4-2$\leftrightarrow$3-51-4-2) move.
This implies that some subsequence of $\sigma$ is an occurrence of one of these eight patterns.
 First, assume that $\sigma_i \sigma_j \sigma_k \sigma_{k+1} \sigma_l$ is an occurrence of the pattern 2-4-15-3 in $\sigma$.
 This means that $\sigma_k<\sigma_i<\sigma_l<\sigma_j<\sigma_{k+1}$ in numerical order.
 However, this implies that the subsequence $\sigma_j\sigma_k \sigma_{k+1}\sigma_l$ is an occurrence of the pattern 3-14-2 in~$\sigma$, contradicting our assumption that~$\sigma$ is a Baxter permutation.
If~$\sigma$ contains an occurrence of one of the other seven patterns, then we similarly show that~$\sigma$ is not a Baxter permutation.
We conclude that the unique permutation mapping to~$D$ under~$\rho$ and satisfying the properties of Lemma \ref{lemma:wallcond} is the Baxter permutation~$\sigma$. 
\end{proof}

\begin{lemma}\label{lemma:arrow}
Let $D$ be a diagonal rectangulation with Baxter poset $P$ naturally embedded in the plane.  
If a linear extension $\sigma$ of $P$ respects the arrows of $P$ then~$\sigma$ satisfies the properties of Lemma \ref{lemma:wallcond}.
\end{lemma}

\begin{proof}
To show that $\sigma$ satisfies the properties of Lemma \ref{lemma:wallcond}, we will show that~$\sigma$ satisfies these properties for each possible configuration of rectangles adjacent to the wall.

First assume that on at least one side of the wall $W$ there is only one adjacent rectangle. 
Let $W$ be a horizontal wall with a single rectangle, rectangle $r_1$, below $W$ and sequence of rectangles $r_2,\ldots,r_l$ above $W$.
For all $i\in \{1,\ldots,l-1\}$, an interior point of rectangle $i$ is strictly below and left of an interior point of rectangle $i+1$. 
Thus, by the definition of the adjacency poset and Theorem \ref{thm:adjposet}, we have that $r_1<_P r_2<_P\cdots<_Pr_l$.
If $W$ is horizontal with a single rectangle, rectangle $r_l$, above $W$ and sequence of rectangles $r_1,\ldots,r_{l-1}$ below $W$, then we reach the same conclusion.
In either case, in $P$, the labels of the rectangles adjacent to $W$ form a chain and, in this chain, all labels of rectangles below~$W$ precede all labels of rectangles above $W$. 
When $W$ is a vertical wall with a single rectangle either left of or right of $W$, the argument is the same.
In these cases, we conclude that the labels of rectangles adjacent to $W$ form a chain in $P$ and the labels of rectangles left of $W$ precede the labels of rectangles right of $W$ in this chain.
Thus every linear extension of~$P$ satisfies the properties of Lemma \ref{lemma:wallcond} for walls that are adjacent to exactly one rectangle on at least one side.

Now assume that on both sides of the wall $W$ there are at least two adjacent rectangles. 
We will prove the claim that if $W$ is a horizontal wall, then the labels of rectangles adjacent to $W$ form a subset of the labels adjacent to some region of $P$.
Let $W$ be horizontal and, as illustrated in the left diagram of Figure \ref{fig:wall}, label from left to right the rectangles adjacent to and below $W$ with the sequence $b_1,\ldots,b_i$.
Label the rectangles adjacent to and above~$W$, again from left to right, $a_1,\ldots,a_j$.
Since $D$ is diagonal and rectangles~$b_1$ and~$a_1$ are the leftmost rectangles adjacent to~$W$, these rectangles form the configuration shown in Diagram~$(i)$ of Figure \ref{fig:adjcovers}.
Thus, by Theorem~\ref{thm:adjcovers}, we have that $b_1\lessdot_P a_1$.
If $a_1<_P b_2$, then there exists a sequence of $x_k$s such that $a_1\lessdot_P x_1\lessdot_P \cdots \lessdot_P x_l\lessdot b_2$.
Since $b_1\lessdot_P a_1$, and $b_2<_P a_j$, for each $k\in [l]$ we have that $b_1<_P x_k<_P a_j$.
Thus each rectangle~$x_k$ is contained either in the region above $W$ and left of the line containing the left side of rectangle~$a_j$ or below $W$ and right of the line containing the right side of rectangle~$b_1$.
But in $P$, no rectangle in the first of these regions covers a rectangle in the second of these regions.
We see by this contradiction that $a_1 \nless_P b_2$.
Since $b_1<_P b_2$ and $a_1 \nless_Pb_2$, there exists some $c$ such that $b_1\lessdot_Pc$ and $c \neq a_1$.
By Theorem~\ref{thm:adjcovers}, rectangle $c$ is adjacent to the right side of rectangle $b_1$.
Since rectangles $b_1$, $a_1$ and~$c$ form a configuration shown in Diagram (ii) or (iii) of Figure \ref{fig:intervalpf}, we have that $a_1 \vee c = a_j$ (as shown in the proof of Theorem \ref{thm:adjposet}).
This implies that $b_1$ and~$a_j$ are contained in a shared region $R$ of the embedded poset.
Observe that for each $k \in [i]$, the lower-left vertex of rectangle $b_k$ is strictly below and left of the upper-right vertex of rectangle $b_i$ so $b_k<_P b_i<_P a_j$.
Similarly, for each $l \in [j]$, we have that $a_l <_P a_j$.

\begin{figure}
\begin{tabular}{c c c c}
\adjustbox{valign=c}{\begin{tikzpicture}
\draw(0,0)--(6,0);
\draw(0,-1)--(0,1);
\draw(6,-1)--(6,1);
\draw(2.8,0)--(2.8,1);
\draw(3.2,0)--(3.2, -1);
\draw (2,0)--(2,1);
\draw(4,0)--(4,-1);
\draw(1,0)--(1,1);
\draw (5,0)--(5,-1);
\draw[fill] (4.25, -.5) circle (.2ex);
\draw[fill] (4.5, -.5) circle (.2ex);
\draw[fill] (4.75,-.5) circle (.2ex);

\draw[fill] (1.75, .5) circle (.2ex);
\draw[fill] (1.5, .5) circle (.2ex);
\draw[fill] (1.25, .5) circle (.2ex);

\node at (.5, .5) {$a_1$};
\node at (2.4, .5) {$a_{j-1}$};
\node at (4.4, .5) {$a_j$};
\node at (1.6, -.5) {$b_1$};
\node at (3.6, -.5) {$b_2$};
\node at (5.5, -.5) {$b_i$};
\end{tikzpicture}}

&  & & \hspace{.3in}

 \adjustbox{valign=c}{\begin{tikzpicture}
\draw(0,0)--(.5,.5);
\draw (0,0)--(-.5,.5);
\draw[dashed] (.5,.5)--(1,1)--(1,3)--(0,4);
\draw[dashed] (-.5,.5)--(-2,2)--(0,4);
\draw[dashed] (-1.5, 1.5)--(-1.2,2)--(-1.5, 2.5);
\draw[dashed] (-.5, .5)--(-.5, 3.5);

\draw[fill] (0,0) circle (.2ex);
\draw[fill] (.5,.5) circle (.2ex);
\draw[fill] (-.5,.5) circle (.2ex);
\draw[fill] (0,4) circle (.2ex);
\draw[fill] (-2,2) circle (.2ex);
\draw[fill] (-1.5, 1.5) circle (.2ex);
\draw[fill] (-1.2, 2) circle (.2ex);
\draw[fill] (-1.5, 2.5) circle (.2ex);
\draw[fill] (-.5, 3.5) circle (.2ex);

\node at (0,-.3) {$b_1$};
\node at (-.7, .4) {$a_1$};
\node at (.7, .4) {$c$};
\node at (0,4.3) {$a_j$};
\node at (-2.2,2) {$a_l$};
\node at (-1, 2) {$d$};
\node at (0.2,2) {$R$};
\node at (-1.55,2) {$R'$};
 \end{tikzpicture}}
\end{tabular}
\caption{Illustrations for the proof of Lemma \ref{lemma:arrow}.}\label{fig:wall}
\end{figure}
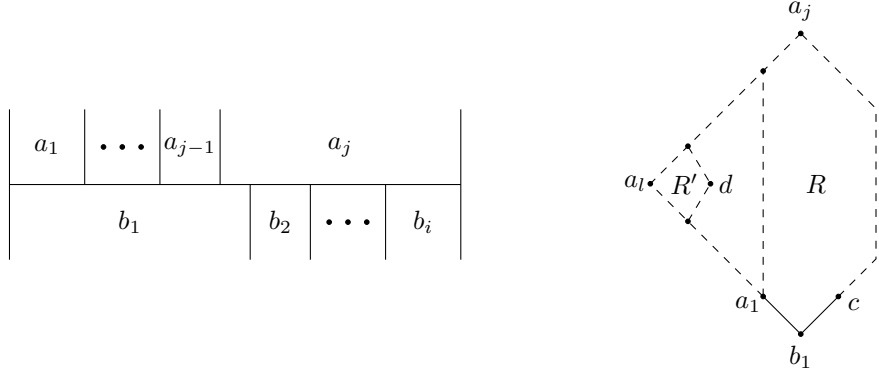

For a contradiction, assume that there exists a label of a rectangle adjacent to~$W$ that is not contained in the boundary of $R$.
We consider the case in which some~$a_l$ is not contained in the boundary of $R$, as illustrated in the right diagram of Figure \ref{fig:wall}.
Since $a_l<b_1$ in numerical order, $a_l$ is contained in the left connected component of the interval $(b_1, a_j)$.
Since $a_l$ is not contained in the left boundary of $R$, the element $a_l$ is contained in the left boundary of some other region, $R'$.
 Let $d$ denote an element contained in the right boundary of $R'$.
 The planarity of the embedding of $P$ implies that  $d$ satisfies $b_1<_P d$.
 Thus $d \nless_P b_1$, implying that that no interior points of rectangle $d$ are strictly left of and below the upper-right corner of rectangle~$b_1$.
 Additionally, $a_l \nless_P d$ so no interior points of rectangle $d$ are strictly right of and above the lower-left corner of rectangle $a_l$.
 Since $d$ and $a_l$ are contained respectively in the right and left boundaries of $R'$, we have that $a_l<d$ in numerical order.
 This implies that rectangle $d$ is contained in the section of the diagonal rectangulation $D$ below the horizontal line containing $W$ and right of the vertical line containing the right side of rectangle~$b_1$.
 Thus $b_1<d$ in numerical order.
 However, this contradicts the assumption that~$P$ is embedded naturally in the plane. 
 We conclude that each $a_l$ for $l\in [j]$ is contained in the left boundary of $R$.
 A similar argument demonstrates that each $b_k$ for $k \in [i]$ is contained in the right boundary of $R$.
 Thus, the claim holds.
 
 Since $W$ is horizontal, $b_1-1=a_j$, implying that the arrow of $R$ points to the left.
 By assumption, $\sigma$ respects the arrows of $R$ so each $b_k$ occurs before every $a_l$ in $\sigma$, i.e. for every horizontal wall, $\sigma$ satisfies the first condition of Lemma \ref{lemma:wallcond}.
 
 A virtually identical argument demonstrates that if $W$ is vertical, and on both sides of $W$ there are at least two adjacent rectangles, then $\sigma$ satisfies the second condition of the lemma.
 \end{proof}

\begin{proof}[Proof of  Theorem \ref{thm:baxter}]
Let $P$ be a Baxter poset, $X$ be the set of linear extensions of~$P$ that respect the arrows of~$P$ and let $\sigma$ be the Baxter permutation that is a linear extension of $P$.
By Lemma \ref{lemma:xnonempty}, the Baxter permutation $\sigma$ is in $X$.
By Lemma~\ref{lemma:arrow}, each element of $X$ satisfies the properties given in Lemma \ref{lemma:wallcond}.
However, by Lemma~\ref{lemma:prop}, only one linear extension of $P$ satisfies these properties so ${X=\{\sigma\}}$.
\end{proof}

\end{document}